\documentclass[review]{amsart}
\usepackage{amsmath,amssymb,color}
\usepackage{lipsum}
\usepackage{amsfonts}
\usepackage{graphicx}
\usepackage{epstopdf}
\usepackage{algorithmic}
\usepackage{hyperref}
\ifpdf
  \DeclareGraphicsExtensions{.eps,.pdf,.png,.jpg}
\else
  \DeclareGraphicsExtensions{.eps}
\fi

\newcommand{\inner}[3][]{\langle#2,#3\rangle_{#1}}

\newcommand{\norm}[2][]{\|{#2}\|_{{#1}}}

\newcommand{\snorm}[2][]{|{#2}|_{{#1}}}

\newcommand{\abs}[1]{|#1|}

\newcommand{\EE}{{\mathbf E}}
\newcommand{\dd}{ { \mathrm{d}} }
\newcommand{\ee}{ { \mathrm{e}} }
\DeclareMathOperator*{\Tr}{\mathrm{Tr}}

\newcommand{\pt}{\partial}
\newcommand{\cT}{{ \mathcal T}}

\newcommand{\cL}{{ \mathcal L}}
\newcommand{\cF}{{ \mathcal F}}
\newcommand{\cD}{{ \mathcal D}}

\newcommand{\IR}{{ \mathbf R}}

\newcommand{\IE}{{ \mathbf E}}
\newcommand{\HS}{\mathrm{HS}}

\newcommand{\ts}{k}

\newcommand{\Rk}{R_{\ts,h}}

\newcommand{\Wj}{\Delta W^j}

\newcommand{\Xbe}{X_{h}}
\newcommand{\Xbej}{\Xbe^j}
\newcommand{\Xben}{\Xbe^n}
\newcommand{\betac}{c_1}
\newcommand{\dH}{\dot{H}}

\newcommand{\HSAQ}{\|A^{1/2}Q^{1/2}\|_{\HS}}

\newcommand{\stig}[1]{\textcolor{blue}{#1}}
\newtheorem{thm}{Theorem} [section]

\newtheorem{lem} [thm] {Lemma}
\newtheorem{prop} [thm] {Proposition}
\newtheorem{rem} [thm] {Remark}
\newtheorem*{assumption*}{Assumption}

\numberwithin{equation}{section}  

\title[Finite element approximation of the Cahn--Hilliard--Cook
equation]{Strong convergence of a fully discrete finite element
  approximation of the stochastic Cahn--Hilliard
  equation} 

\author[D.~Furihata]{Daisuke Furihata}

\address{Cybermedia Center, Osaka University, 1-32 Machikaneyama,
  Toyonaka, Osaka 560-0043, Japan
  }

\email{furihata@cmc.osaka-u.ac.jp}

\author[M.~Kov\'acs]{Mih\'aly Kov\'acs}\thanks{M. Kov\'acs was supported by Marsden Fund of the Royal Society of New Zealand project number UOO1418. }

\address{Department of Mathematical Sciences, Chalmers University of Technology and 
  University of Gothenburg, SE--412 96 Gothenburg, Sweden}

 \email{mihaly@chalmers.se}

\author[S. Larsson]{Stig Larsson}

\address{Department of Mathematical Sciences, Chalmers University of Technology and 
  University of Gothenburg, SE--412 96 Gothenburg, Sweden}

 \email{stig@chalmers.se}

\author[F.~Lindgren]{Fredrik Lindgren}\thanks{F. Lindgren was supported by JSPS KAKENHI grant number 15K45678 and by The Swedish Foundation for International Cooperation in Research and Higher Education, STINT through the Capstone Award.}

\address{Cybermedia Center, Osaka University, 1-32 Machikaneyama,
  Toyonaka, Osaka 560-0043, Japan
  }

\email{fredrik.lindgren1979@gmail.com}

\keywords{stochastic partial differential equation;
  Cahn--Hilliard--Cook equation; additive noise; Wiener process; Finite
  element method, Euler method; time discretization; strong
  convergence} 

\subjclass[2000]{60H15, 60H35, 65C30}

\begin{document}
\maketitle

\begin{abstract}
  We consider the stochastic Cahn--Hilliard equation driven by
  additive Gaussian noise in a convex domain with polygonal boundary
  in dimension $d\le 3$. We discretize the equation using a standard
  finite element method in space and a fully implicit backward Euler
  method in time. By proving optimal error estimates on subsets of the
  probability space with arbitrarily large probability and
  uniform-in-time moment bounds we show that the numerical solution
  converges strongly to the solution as the discretization parameters
  tend to zero. 
\end{abstract}

\section{Introduction}

Let $\mathcal{D}\subset \mathbb{R}^d$, $d\le 3$, be a convex spatial
domain with polygonal boundary $\partial\cD$ and consider the
stochastic Cahn--Hilliard equation, also known as the
Cahn--Hilliard--Cook equation \cite{MR1870315,Cook,MR1359472}, written
in the abstract It\^o form
\begin{equation}\label{eq:sac}
\dd X+A(A X+f(X))\,\dd t=\dd W,~t\in(0,T];\quad X(0)=X_0,
\end{equation}
where $-A$ is the Laplacian with homogeneous Neumann boundary
conditions and where $\{W(t)\}_{t\ge 0}$ is an $H:=L_2(\cD)$-valued $Q$-Wiener
process on a filtered probability space
$(\Omega, \cF, \mathbf{P}, \{\cF_t\}_{t\ge 0})$ with respect to the
normal filtration $ \{\cF_t\}_{t\ge 0}$. In order to avoid additional
technical difficulties we assume that the initial-value $X_0$ is
deterministic.  For the weak (and mild) solution $X$ (see, Theorem \ref{weaksol})  to preserve mass, we assume that
the average $ \abs{\cD}^{-1}\int_\cD W(t) \,\dd x=0$ for all
$t\geq 0$.

The nonlinear function $f$ is assumed to be of the form $f=F'$, where
$F$ has the following structural property:
\begin{equation}\label{eq:F}
\text{$F$ is a polynomial of degree $4$ with leading term $c_0s^4$
  where $c_0>0$.} 
\end{equation}
A typical example is $F(s)=\frac14(s^2-1)^2$, which is a double
well potential. Note that $f$ is only locally Lipschitz and does not
satisfy a linear growth condition. We also note that the restriction
on the polynomial degree of $F$ comes from the fact that we allow
$d=3$. For $d=1,2$ the exponent 4 in \eqref{eq:F} may be replaced by
any even integer larger than or equal 4 and the arguments of the paper
are still valid with trivial changes.  It is easy to see that
\eqref{eq:F} implies the following dissipativity property, with
$\inner{\cdot}{\cdot}$, $\norm{\cdot}$ denoting the scalar product and
norm in $H=L_2(\cD)$,
\begin{equation}\label{eq:diss1}
\langle f(v),v\rangle\geq -C_0, \quad v\in L_4(\cD),
\end{equation}
for some $C_0$.  It also follows that $F''(s)\geq
-\betac^2$ for some $\betac$, which yields
\begin{equation}\label{eq:taylorf}
F(x)-F(y)\leq f(x)(x-y)+\tfrac12 \betac^2(x-y)^2, \quad x,y\in\IR.
\end{equation}
Finally, as $f$ is a polynomial of degree $3$ we have, for some $C>0$,
that
\begin{equation}\label{eq:loclip}
|f(x)-f(y)|\leq C(1+x^2+y^2)|x-y|,\quad x,y\in\IR.
\end{equation}

It is not hard to see that, due to the homogeneous Neumann boundary
conditions and since the average of $W$ equals 0, it
follows that $X$ preserves the average (the total mass), that is,
$\abs{\cD}^{-1}\int_\cD X(t) \,\dd x =\abs{\cD}^{-1}\int_\cD X_0 \,\dd
x$, cf.\ Remark~\ref{rem:preservemass}.
Note that for $s_0\in \mathbb{R}$ the function $\tilde{F}(s)=F(s+s_0)$
also has the structural property \eqref{eq:F}. Therefore, one
can employ a change of variables
$X\to X-\abs{\cD}^{-1}\int_\cD X_0 \,\dd x$, and hence we will assume
that the average $\abs{\cD}^{-1}\int_\cD X_0 \,\dd x=0$.

We fix a finite time horizon $T>0$ and for $N\in \mathbb{N}$
consider the fully implicit finite element method
\begin{equation}\label{eq:BE}
\begin{aligned}
\Xbej-\Xbe^{j-1}+kA_h^2\Xbej+kA_hP_hf(\Xbej)
&=P_h\Wj,\quad j=1,2,\ldots,N,\\
 X^0_h&=P_hX_0.
\end{aligned}
\end{equation}
Here $k=T/N$ is the time-step, $t_j=jk$,
$\Delta W^{j}=W(t_j)-W(t_{j-1})$, $-A_h$ is the discrete
Laplacian, and $P_h$ is the orthogonal projector onto the finite
element space $S_h$ with mesh size $h>0$; for more details on the
finite element method, see Section~\ref{subsec:fem}.  It is easy to
see that also $X_h^j$ preserves the mass, cf.\
  Remark~\ref{rem:preservemass}.  An implementation based on
  the open source finite element software FEniCS can be found in
  \texttt{http://www.math.chalmers.se/\%7estig/code/chc.py}.
 
The main result of the paper, Theorem~\ref{thm:main}, asserts that if
the operator composition $A^{\frac{1}{2}}Q^{\frac{1}{2}}$ is
Hilbert--Schmidt and the initial data is regular enough; that is, for
some $L>0$,
\begin{equation*}
|X_h^0|_{1}+\cF(X_h^{0})+|A_hX_h^{0}+P_hf(X_h^{0})|_1+|X_0|_1\le L\text{ for all }h>0,
\end{equation*}
where $\cF(v)=\int_{\cD}F(v)\,\dd x$ and
$|v|_1=\|A^{\frac{1}{2}}v\|=\norm{\nabla v}$, then
\begin{align*}
\lim_{h,k\to 0}\EE\sup_{0\le n\le N}\|X(t_n)-\Xbe^n\|^2=0.  
\end{align*}

The key result used in the proof is a maximal type moment bound on
$|\Xbej|_1$, which is established in Theorem~\ref{thm:H1moment} after
bootstrapping arguments.  There are various difficulties in
the proofs that are partly due to the finite element method.  First,
the finite element method is based on approximating the operator $A$
and not $A^2$. This is because the standard finite element functions
belong only to the domain of $A^{\frac12}$ but are not more
regular. Loosely speaking this means that $A^2_h\neq (A^2)_h$, which
makes already the deterministic finite element analysis more
challenging. Second, the presence of the finite element projection
$P_h$ in front of the semilinear term destroys some of the
dissipativity properties of $f$. While $f$ enjoys the dissipativity
property \eqref{eq:diss1}, and even
\begin{equation*}
\langle A^{\frac12}f(v),A^{\frac12}v\rangle 
=\langle \nabla f(v),\nabla v\rangle 
\geq-c|v|_1^2,
\end{equation*}
we only have 
\begin{equation*}
\langle P_hf(v_h),v_h\rangle
=\langle f(v_h),v_h\rangle
\geq -C_0,\quad v_h\in S_h,
\end{equation*}
and unfortunately
\begin{equation*}
\langle A_h^{\frac12}P_hf(v_h),A_h^{\frac12}v_h\rangle
=\langle A^{\frac12}P_hf(v_h),A^{\frac12}v_h\rangle
\ngeq-c|v_h|_1^2,\quad v_h\in S_h.
\end{equation*}
Because of the latter we can only establish a non-uniform moment
bound on $\|\Xbej\|$ in Lemma~\ref{lem:qest}.  As
\begin{equation*}
\langle Af(x),Ax\rangle\ngeq-c\|Ax\|^2
\text{ and }\langle A_hP_hf(v_h),A_hv_h\rangle\ngeq-c\|A_hv_h\|^2,
\end{equation*}
the proof of the main moment bound in Theorem \ref{thm:H1moment} is
rather tedious and not entirely straightforward.  In the proof we
bound the discrete version of the Ljapunov functional for the original
deterministic problem, see \eqref{eq:Lyapunovest}. Having maximal-type
moment bounds at hand we use the mild formulation of both
\eqref{eq:sac} and \eqref{eq:BE} to establish pathwise error bounds on
subsets of the probability space with large probability in
Theorem~\ref{thm:pconv}. This turns out to be sufficient, together
with some moment bounds, to show strong convergence of the numerical
scheme in Theorem \ref{thm:main}. Our method of proof does not give
rates for the strong convergence.


While strong convergence results for numerical schemes for SPDEs with
globally Lipschitz coefficients, or at least some sort of linear
growth condition, are abundant there are only few results on strong convergence of
discretization schemes for SPDEs with superlinearly growing
coefficients \cite{2016arXiv160105756B, GyM,MR3498982, JP,
  KLLAC, 2015arXiv151003684K, Kur}.  Furthermore, these papers
dominantly establish strong convergence of various numerical schemes
with no rate given (with a few exceptions), under some sort of global
monotonicity assumption on the drift term which is not valid for the
Cahn--Hilliard--Cook equation \eqref{eq:sac}.

 The analysis of numerical methods for SPDEs without a global
  monotonicity assumption is even less explored. For the
  Cahn--Hilliard--Cook equation \eqref{eq:sac} studied in the present
  paper, convergence (with rates) in probability is established for a
  finite difference scheme in \cite{CW}. In \cite{2014arXiv1401.0295H}
  strong convergence with rates is established for the spatial
  spectral Galerkin approximation (no time discretization) for
  \eqref{eq:sac} and the stochastic Burgers equation driven by trace
  class noise in spatial dimension $d=1$. The analysis is based on a
  general perturbation result and exponential integrability properties
  of the approximation process. Strong convergence of the finite
  element method without rate and without time discretization is
  proved in \cite{KLM,MR3273327} under stronger assumptions on the
  noise than in the present paper. There it is required that the
  operator composition $A^{\frac{\gamma}{2}}Q^{\frac12}$ for
  $\gamma>1$ is Hilbert--Schmidt, while here we only require this with
  $\gamma=1$.  Therefore, the present work can be viewed as the
  (non-trivial) extension of \cite{KLM,MR3273327} to a strongly
  convergent fully discrete scheme, still without a strong rate, but
  with improvements on the regularity requirement on the noise.  Both
  here and in \cite{KLM,MR3273327} the strategy is based on proving a
  priori moment bounds with large exponents and in higher order norms
  using energy arguments and bootstrapping followed by a pathwise
  Gronwall argument in the mild solution setting.  Finally, in
  connection to the Cahn--Hilliard--Cook equation we note that in
  \cite{MR3031667,MR2846757} the linearized Cahn--Hilliard--Cook
  equation is treated numerically which requires a significantly
  simpler analysis.

As further related work on numerical approximation of SPDEs without a global
monotonicity assumption we mention the pathwise convergence of
a spectral Galerkin method for the stochastic Burgers equation
studied in \cite{Blomker_Jentzen,Blomker_et_al}, while for the same
equation convergence in probability is established in \cite{P2001} for
the Backward Euler method. The stochastic Navier--Stokes equation is
considered in \cite{MR3081484,MR3022227}, in particular, in
\cite{MR3022227} the authors obtain a result similar to our
Theorem~\ref{thm:pconv} (stated in a slightly different form). 
Finally, we mention the recent work
\cite{2016arXiv160402053H}, where strong convergence is proved,
without rate, for a spectral nonlinearity-truncated accelerated
exponential Euler-type approximation for the stochastic
Kuramoto--Sivashinsky equation driven by space-time white noise in
spatial dimension $d=1$, an equation rather similar in structure to
the Cahn--Hilliard--Cook equation.

The paper is organized as follows. In Section~\ref{sec:prelim} we
collect some background material from stochastic and functional
analysis and introduce the finite element method in
Subsection~\ref{subsec:fem}. In Section~\ref{sec:eur} some known
results on the existence, uniqueness and regularity on the solution of
 \eqref{eq:sac} are recalled. Section~\ref{sec:momb} contains
moment bounds for the numerical solution, in particular, it contains
the main technical result of the paper, Theorem \ref{thm:H1moment}. In
Section~\ref{sec:conv} we prove a new error estimate for the
derivative of the error in the spatial semidiscretization of the
linear deterministic Cahn--Hilliard equation, \eqref{eq:dette1} in
Lemma~\ref{lem:dete}. Then we proceed to prove a pathwise error bound
in Theorem~\ref{thm:pconv} and the main strong convergence result
Theorem~\ref{thm:main}.

\section{Preliminaries}\label{sec:prelim}

\subsection{Norms and operators} Throughout the paper we will use
various norms for linear operators on a Hilbert space $H$ where the
latter is endowed with inner product $\langle\cdot,\cdot\rangle$ and
norm $\norm{\cdot}$. We denote by $\mathcal{L}(H)$, the space of
bounded linear operators on $H$ with the usual operator norm also denoted
by $\norm{\cdot}$. If for a selfadjoint positive semidefinite operator $T\in\cL(H)$,
the sum
\begin{align*}
\Tr T:=\sum_{k=1}^\infty\langle Te_k,e_k\rangle<\infty  
\end{align*}
for an orthonormal basis (ONB) $\{e_k\}_{k\in \mathbb{N}}$ of $H$,
then we say that $T$ is trace class. In this case $\Tr T$, the trace
of $T$, is independent of the choice of the ONB. If for an operator
$T\in\cL(H)$, the sum
\begin{align*}
\|T\|_{\HS}^2:=\sum_{k=1}^\infty\|Te_k\|^2<\infty  
\end{align*}
for an ONB $\{e_k\}_{k\in \mathbb{N}}$ of $H$, then we say that $T$ is
Hilbert--Schmidt and call $\|T\|_{\HS}$ the Hilbert--Schmidt norm of
$T$.  The Hilbert--Schmidt norm of $T$ is independent of the choice of
the ONB. We have the following well-known properties of the trace and
Hilbert--Schmidt norms, see, for example, \cite[Appendix C]{DPZ},
\begin{align}
\label{eq:hs}
\|T\|&\le \|T\|_{\HS},\quad \|TS\|_{\HS}\le \|T\|_{\HS}\|S\|,
\quad\|ST\|_{\HS}\le \|S\|\,\|T\|_{\HS},
\\
\label{eq:tr}
\Tr Q&=\|Q^{\frac12}\|_{\HS}^2=\|T\|^2_{\HS}=\|T^*\|_{\HS}^2,
\quad\text{ if $Q=TT^*$.}
\end{align}

Next we introduce spaces and norms associated with the operator
  $A$, the negative of the Neumann Laplacian. Let
$\cD \subset \mathbb{R}^d$, $d =1,2,3$, be a bounded convex domain
with polygonal boundary $\pt \cD$.  We denote by $\norm[L_p]{\cdot}$
the standard norm in $L_p(\cD)$. In particular, we define
$H = L_2(\cD)$ with its standard inner product $\inner{\cdot}{\cdot}$
and norm $\norm{\cdot}$, and
\begin{equation*}
\dot{H} = \Big \lbrace v \in H : \int_\cD v \,\dd x = 0 \Big \rbrace.
\end{equation*}
Let $P \colon H \to \dot{H}$ define the orthogonal projector. Then
$ (I-P)v = \abs{\cD}^{-1}\int_\cD v \,\dd x$ is the average of $v$.
We also denote by $H^k(\cD)$ the standard Sobolev space. We define
$A = -\Delta$, the negative of the Neumann Laplacian with
domain of definition
\begin{equation*}
D(A) = \Big \lbrace v \in H^2(\cD): \frac{\pt v}{\pt n} = 0\
\text{on}\ \pt \cD \Big \rbrace.
\end{equation*}
Then $A$ is a positive definite, selfadjoint, unbounded, linear
operator on $\dot{H}$ with compact inverse.  When extended to $H$ as
$Av=APv$ it has an orthonormal eigenbasis $\lbrace \varphi_j
\rbrace_{j=0}^\infty$ with corresponding eigenvalues $\lbrace
\lambda_j \rbrace_{j=0}^\infty$ such that
\begin{equation*}
0 = \lambda_0 < \lambda_1 \le \lambda_2 \le \cdots \le \lambda_j
\le \cdots ,
\quad \lambda_j \to \infty,
\end{equation*}
see, for example, \cite[Section 7.2]{Dav}. The first
eigenfunction is constant, $\varphi_0 = \abs{\cD}^{-\frac{1}{2}}$. 

We define 
\begin{align}
\label{eq:dotnorm}
& \snorm[\alpha]{v} 
= \Big( \sum_{j=1}^\infty 
\lambda_j^\alpha \abs{\inner{v}{\varphi_j}}^2 \Big)^{\frac{1}{2}},
\quad \inner[\alpha]{v}{w}
= \sum_{j=1}^\infty 
\lambda_j^\alpha \inner{v}{\varphi_j}\inner{w}{\varphi_j},
\quad \alpha \in\mathbb{R},\\
\label{eq:nondotnorm}
& \norm[\alpha]{v} = \big( \snorm[\alpha]{v}^2 +
  \abs{\inner{v}{\varphi_0}}^2\big)^{\frac{1}{2}} ,
\quad \alpha \ge0,
\end{align}
and corresponding spaces, for $\alpha\ge0$,
\begin{equation*}
\dot{H}^\alpha = D(A^{\frac{\alpha}{2}})
= \Big\lbrace v\in \dot{H}: \snorm[\alpha]{v} < \infty \Big\rbrace,
\quad H^\alpha = \Big \lbrace v\in H: \norm[\alpha]{v} < \infty \Big \rbrace.
\end{equation*}
For negative order $-\alpha<0$ we define $\dH^{-\alpha}$ by
taking the closure of $\dH$ with respect to $\snorm[-\alpha]{\cdot}$.
For integer order $\alpha=k=1,2$, 
the norm $\norm[k]{\cdot}$ is equivalent on $H^k$ to the
standard Sobolev norm $\norm[H^k(\cD)]{\cdot}$. More precisely, 
\begin{align*} 
\norm[1]{v} &= \big(\snorm[1]{v}^2 + \abs{\inner{v}{\varphi_0}}^2 \big)^{\frac12}
=\big(\norm{\nabla v}^2 + \abs{\inner{v}{\varphi_0}}^2 \big)^{\frac12}, 
\\
\norm[2]{v} &= \big(\snorm[2]{v}^2 + \abs{\inner{v}{\varphi_0}}^2 \big)^{\frac12}
=\big(\norm{\Delta v}^2 + \abs{\inner{v}{\varphi_0}}^2 \big)^{\frac12}
\end{align*}
are equivalent to the standard norms $\norm[H^k(\cD)]{v}$, $k=1,2$,
by the Poincar\'e inequality and the regularity estimate for the
elliptic Neumann problem.  

We recall the fact the operator $-A^2$ is the infinitesimal generator
of an analytic semigroup $E(t)=\ee^{-tA^2}$ on $H$,
\begin{equation}\label{chsg}
\begin{split}
E(t)v= \ee^{-tA^2} v
& = \sum_{j=0}^\infty \ee^{-t \lambda_j^2}\inner{v}{\varphi_j}\varphi_j
   = \sum_{j=1}^\infty \ee^{-t \lambda_j^2}\inner{v}{\varphi_j}\varphi_j
   + \inner{v}{\varphi_0}\varphi_0
\\ & 
= P\ee^{-tA^2}v+ (I-P)v.
\end{split}
\end{equation}
By expansion in the eigenbasis of $A$ and using Parseval's identity we
easily obtain 
\begin{equation}\label{eq:analytic}
\norm{A^\alpha E(t)v} 
\le C t^{-\frac{\alpha}{2}}  \norm{v},
\quad v\in H,\ \alpha \ge 0.
\end{equation}
and
\begin{align}
  \label{eq:analytic1}
  \Big( \int_0^t s^{2j}\|A^{2j+1}E(s)v\|^2\,\dd s \Big)^{1/2}
  \le C\|v\|,
\quad v\in H, \ j=0,1,2\dots. 
\end{align}
Here $C$ depends on $\alpha$ and $j$, respectively.

\subsection{The finite element method} \label{subsec:fem} 
Let $\lbrace \cT_h \rbrace_{h>0}$ denote a family of regular
triangulations of $\cD$ with maximal mesh size $h$.  Let $S_h$ be the
space of continuous functions on $\cD$, which are piecewise
polynomials of degree $\le1$ with respect to $\cT_h$. Hence,
$S_h \subset H^1$. We also define $\dot{S}_h = PS_h$; that is,
\begin{equation*}
\dot{S}_h = \Big \lbrace v_h \in S_h: \int_\cD v_h \,\dd x = 0 \Big \rbrace.
\end{equation*}
The space $\dot{S}_h$ is introduced only for the purpose of theory but
not for computation. Now we define the discrete Laplacian
$-A_h \colon S_h \to S_h$ by
\begin{equation*}
\inner{A_h v_h}{w_h} = \inner{\nabla v_h}{\nabla w_h},
\quad \forall v_h \in S_h,\, w_h \in S_h.
\end{equation*}
The operator $A_h$ is selfadjoint, positive definite on $\dot{S}_h$,
positive semidefinite on $S_h$, and $A_h$ has an orthonormal
eigenbasis $\lbrace \varphi_{h,j} \rbrace_{j=0}^{N_h}$ with
corresponding eigenvalues $\lbrace \lambda_{h,j} \rbrace_{j=0}^{N_h}$.
We have
\begin{equation*}
0 = \lambda_{h,0} < \lambda_{h,1} \le \cdots \le \lambda_{h,j} \le
\cdots \le \lambda_{h,N_h}
\end{equation*}
and $\varphi_{h,0} = \varphi_0 = \abs{\cD}^{-\frac{1}{2}}$.
Moreover, we define $E_h(t)=\ee^{-tA_h^2} \colon S_h \to S_h$ by
  \begin{equation*}
  E_h(t)v_h=\ee^{-t A_h^2} v_h
  = \sum_{j=0}^{N_h} \ee^{-t \lambda_{h,j}}
   \inner{v_h}{\varphi_{h,j}}\varphi_{h,j}
\end{equation*}
and the orthogonal projector $P_h \colon H \to S_h$ by
\begin{equation}\label{P_h}
\inner{P_h v}{w_h} = \inner{v}{w_h}\quad \forall v \in H,\, w_h \in S_h.
\end{equation}
Clearly, $P_h \colon \dot{H} \to \dot{S}_h$ and
\begin{align*} 
E_h(t)P_h v = PE_h(t) P_h v + (I-P)v.  
\end{align*} 
Likewise, for its time discrete analog $\Rk^n=(I+\stig{k}\,
A_h^2)^{-n}$, we have 
\begin{align} \label{chsg2}
\Rk^nP_h v = P\Rk^n P_h v + (I-P)v. 
\end{align} 
We have the discrete analogs of \eqref{eq:analytic},
\begin{equation}\label{eq:fem2}
\norm{A_h^\alpha E_h(t)P_hv} 
\le C t^{-\frac{\alpha}{2}} \norm{v}, 
\quad 
\norm{A_h^\alpha \Rk^nP_hv} 
\le C t_n^{-\frac{\alpha}{2}} 
\norm{v},\quad v\in H,\ \alpha \ge 0,
\end{equation} 
where the constants $C$ depend on $\alpha$ but not on $h$ and $k$. Similarly
to \eqref{eq:analytic}, these are proved by expansion in the
eigenbasis of $A_h$ and Parseval's identity.  For example, the first
constant is $C=(\sup_{s\in[0,\infty)}s^\alpha\ee^{-2s})^{\frac12}$.

Finally, we define the Ritz projector $R_h \colon \dot{H}^1 \to \dot{S}_h$ by
\begin{align*}
\inner{\nabla R_h v}{\nabla w_h} = \inner{\nabla v}{\nabla w_h},
\quad \forall v \in \dot{H}^1,\, w_h \in \dot{S}_h.
\end{align*}
We extend it as $R_h \colon {H}^1 \to {S}_h$ by
\begin{align}\label{R_h}
  R_hv=R_hPv+(I-P)v,  \quad v \in {H}^1.
\end{align}
We then have the following bound for $R_hv-v=(R_h-I)Pv$ (cf.\
\cite[Chapt.~1]{Thomeebook})
\begin{align}
  \label{eq:fem1}
\norm{(R_h-I)v}\le Ch^2\norm{Av},  \quad v\in \dot{H}^2.
\end{align}
Finally, we define norms on $\dot{S}_h$, analogous to the norms
  $\snorm[\alpha]{\cdot}$ on $\dH^\alpha$: 
\begin{align}
  \label{eq:discretenorm}
  \snorm[\alpha,h]{v_h}=\norm{A_h^{\alpha/2}v_h}
= \Big( \sum_{j=1}^{N_h}
\lambda_{j,h}^\alpha \abs{\inner{v}{\varphi_{j,h}}}^2 \Big)^{\frac{1}{2}},
\quad v_h\in\dot{S}_h,\ \alpha \in\mathbb{R}. 
\end{align}
The corresponding scalar products
are denoted $\inner[\alpha,h]{\cdot}{\cdot}$.  We note that 
\begin{equation}\label{chc2:a1}
\snorm[1]{v_h} = \norm{A^{\frac{1}{2}} v_h}
= \norm{\nabla v_h} = \norm{A_h^{\frac{1}{2}}v_h}
=\snorm[1,h]{v_h}, \quad v_h \in \dot{S}_h. 
\end{equation}
We assume that $P_h$ is bounded
with respect to the $\dH^1$ norm
\begin{equation}\label{x}
\snorm[1]{P_h v} \le C\snorm[1]{v},\quad v \in {\dH}^1.
\end{equation}
This holds, for example, if the mesh family
$\lbrace \cT_h \rbrace_{h > 0}$ is quasi-uniform.
By combining this with \eqref{chc2:a1}, we obtain
\begin{align}
  \label{chc2:a2}
\norm{A_h^{1/2}P_hv}
=\snorm[1]{P_hv}
\le  C\snorm[1]{v}
=C\norm{A^{1/2}v}. 
\end{align}

\subsection{Useful inequalities} We will use the
Burkholder--Davis--Gundy inequality for It\^o-integrals of the form
$\int_0^t\langle \eta(s),\dd \tilde{W}(s)\rangle$, where $\eta$
  is a predictable $H$-valued stochastic process and $\tilde{W}$ is a
$\tilde{Q}$-Wiener process in $H$.  For this kind of integral,
the Burkholder--Davis--Gundy inequality, \cite[Lemma 7.2]{DPZ}, takes
the form
\begin{equation}\label{eq:rbdg}
\EE\sup_{t\in [0,T]}
\Big|\int_0^t\langle \eta(s),\dd \tilde{W}(s)\rangle\Big|^p
\le C
\EE\Big(
\int_0^{T}\|\tilde{Q}^{\frac12}\eta(s)\|^2\,\dd s
\Big)^{\frac{p}{2}}, 
\quad p\ge 2,
\end{equation}
where $C$ depends on $p$.

Also, if $Y$ is an $H$-valued centered Gaussian random variable
with covariance operator $\tilde{Q}$, then, by \cite[Corollary
2.17]{DPZ}, we can bound its $p$-th moments via its covariance
operator as
\begin{equation}\label{eq:es1a}
\EE\|Y\|^{2p}\le C(\EE\|Y\|^2)^{p}
=C (\Tr \tilde{Q})^{p}=C\|\tilde{Q}^{\frac12}\|_{\HS}^{2p}, 
\quad p\ge1, 
\end{equation}
where $C$ depends on $p$.  In particular, for an  It\^o integral
$Y=\int_s^tR\,\dd W(r)=R(W(t)-W(s))$, where $R$ is a constant, possibly unbounded,
operator on $H$ and $W$ is a $Q$-Wiener process, 
the inequality \eqref{eq:es1a} reads
\begin{equation}\label{eq:es1b}
\EE\Big\|\int_s^tR\,\dd W(r)\Big\|^{2p}\leq C(t-s)^p\|RQ^{1/2}\|_{\HS}^{2p}.
\end{equation} 
If $p=1$, the inequality in \eqref{eq:es1b} becomes an equality with
$C=1$.  The inequality
\begin{equation}\label{eq:psum}
\Big|\sum_{j=K}^Ma_j\Big|^p
\leq |M-K+1|^{p-1}\sum_{j=K}^M|a_j|^p, \quad p\ge1,
\end{equation}
will be frequently utilized; it is a direct consequence of H\"older's
inequality.

\section{Existence, uniqueness and regularity}\label{sec:eur}
Existence, uniqueness, and regularity of weak solutions to \eqref{eq:sac}
has been studied in \cite{MR1359472} with some minor improvements in
\cite{KLM}. Note that here we assume that $X_0$ is deterministic and
that $X_0\in\dH$, so that $X(t)\in\dH$. We summarize the results:
\begin{thm}\label{thm:existenceandregularityofX}
If $\|A^{\frac12}Q^{\frac12}\|_{\HS}<\infty$, 
$|X_0|_1<\infty$, and $T<\infty$, then there is a
unique weak solution $X$ of \eqref{eq:sac}; that is, 
 an adapted $H$-valued
process $X$, which is continuous almost surely and satisfies the equation
\begin{align}\label{weaksol}
\inner{X(t)}{v} - \inner{X_0}{v}
+ \int_0^t \big(\inner{X(s)}{A^2 v}
+ \inner{f(X(s))}{A v}\big)\,\dd s
=  \inner{ W(t)}{v}
\end{align}
almost surely for all $v \in \dot{H}^4=D(A^2)$, $t\in [0,T]$.
Furthermore,
there is $C>0$ such that
\begin{equation}\label{eq:esup}
\mathbf{E}\sup_{t\in [0,T]}|X(t)|_1^2+\mathbf{E}\sup_{t\in [0,T]}\|X(t)\|_{L_4}^4\le C.
\end{equation}
In addition, $X$ is also a mild solution; that is, it satisfies
  the equation
\begin{equation}\label{eq:mildsac}
X(t)=E(t)X_0-\int_0^t E(t-s)Af(X(s))\,\dd s+\int_0^t E(t-s)\,\dd W(s),
\end{equation}
almost surely.
\end{thm} 

We also have pathwise H\"older regularity in time:
\begin{prop}\label{prop:ph}
  Let $\|A^{\frac{1}{2}}Q^{\frac12}\|_{\HS}<\infty$ and
  $|X_0|_1<\infty$. Then, for all $\gamma\in[0,\frac12)$, there is 
  an almost surely finite nonnegative random variable $K$ such
  that, almost surely,
\begin{align*}
\sup_{t\neq s\in [0,T]}\frac{\|X(t)-X(s)\|}{|t-s|^{\gamma}}\le K.  
\end{align*}
\end{prop}

We omit the proof as it is analogous to the proof of
\cite[Proposition 3.2]{KLLAC}.

\section{Moment bounds for the discrete solution}\label{sec:momb}
We start by proving a preliminary moment bound which will be used
later on in a bootstrapping argument. Throughout the proofs,
  $C$ denotes a generic non-negative constant that is independent of
  the discretization parameters $h$ and $k$ and may change from line
  to line.

We recall our assumption that
$X_0\in\dH$, so that $X_h^0=P_hX_0\in\dot{S}_h$ and
hence $X_h^j\in \dot{S}_h$ for $0\le j\le N$.  

\begin{lem}\label{lem:Hmoment}
  Let $p\geq 1$.  If $\|A_h^{-1/2}P_hQ^{1/2}\|_{\HS}\le K$ and
    $|X_{h}^0|_{-1,h}\le L$ for all $h>0$, then there exists $C>0$
    depending on $p,T,K,L$, such that, for all $h,k>0$, 
\begin{align}
\label{eq:min1}
&\EE\Big(\sup_{1\leq j\leq N}|\Xbej|_{-1,h}^{2p}\Big)\leq C,
\\ &
\label{eq:min2}
\EE\Big( 
\sum_{j=1}^N\big(|\Xbej-\Xbe^{j-1}|_{-1,h}^2+k|\Xbej|^2_1\big)\Big)^{p}\leq C.
\end{align}
\end{lem}

\begin{proof}
  Since $\Xbej\in \dot{S}_h$, we may multiply by $A_h^{-1}\Xbej$ in
  \eqref{eq:BE} to get
\begin{align*}
&\frac12\left(|\Xbej|_{-1,h}^2-|\Xbe^{j-1}|_{-1,h}^2+|\Xbej-\Xbe^{j-1}|_{-1,h}^2\right)
+k|\Xbej|^2_{1}+k\langle f(\Xbej),\Xbej\rangle\\
&\qquad=\langle P_h\Wj,A^{-1}_h\Xbej\rangle,
\end{align*}
where we have used the selfadjointness of $A_h$, \eqref{chc2:a1}, and
the identity
\begin{equation}\label{eq:betrick}
\langle X-Y,X\rangle_{-1,h}
=\tfrac12
\big(|X|_{-1,h}^2-|Y|_{-1,h}^2+|X-Y|_{-1,h}^2\big).  
\end{equation}
Next, the dissipativity inequality \eqref{eq:diss1} yields
\begin{equation*}
\begin{aligned}
&\frac12\left(|\Xbej|_{-1,h}^2-|\Xbe^{j-1}|_{-1,h}^2+|\Xbej-\Xbe^{j-1}|_{-1,h}^2\right)+k|\Xbej|^2_1\\
&\qquad\leq
C_0k+\langle P_h \Wj,\Xbej-\Xbe^{j-1}\rangle_{-1,h}+\langle P_h\Wj,\Xbe^{j-1}\rangle_{-1,h}.
\end{aligned}
\end{equation*} 
Furthermore, $\langle P_h \Wj,\Xbej-\Xbe^{j-1}\rangle_{-1,h}\leq
C_{\epsilon}|P_h \Wj|_{-1,h}^2+\epsilon|\Xbej-\Xbe^{j-1}|_{-1,h}^2$. Hence, 
\begin{equation}\label{eq:onestep}
\begin{aligned}
&\frac12\left(|\Xbej|_{-1,h}^2-|\Xbe^{j-1}|_{-1,h}^2\right)
+c|\Xbej-\Xbe^{j-1}|_{-1,h}^2+k|\Xbej|^2_1\\
&\qquad\leq
C\left(k+|P_h\Wj|_{-1,h}^2+\langle P_h\Wj,\Xbe^{j-1}\rangle_{-1,h}\right).
\end{aligned}
\end{equation}
Summing with respect to $j$ in \eqref{eq:onestep}, thus yields
\begin{equation}\label{eq:normsum}
\begin{aligned}
&|\Xbe^n|_{-1,h}^2+\sum_{j=1}^n\left(|\Xbej-\Xbe^{j-1}|_{-1,h}^2+k|\Xbej|^2_1\right)
\leq C\Big(T+ |X^0_h|_{-1,h}^2
\\ & \qquad 
+\sum_{j=1}^n|P_h\Wj|_{-1,h}^2
+\sum_{j=1}^{n}\langle P_h\Wj,\Xbe^{j-1}\rangle_{-1,h}\Big).
\end{aligned}
\end{equation} 
We drop the sum on the left, take the $p$'th power and the supremum
with respect to $n$, and then the expectation, using also
\eqref{eq:psum} repeatedly, to get
\begin{equation}\label{eq:normest}
\begin{aligned}
&\EE \sup_{1 \leq n\leq N}|\Xben|_{-1,h}^{2p}
\leq\EE\sup_{1\leq n\leq N} C\Big\{
T+ |X^0_h|_{-1,h}^2
\\ & \qquad 
+\sum_{j=1}^n|P_h\Wj|_{-1,h}^2 
+\sum_{j=1}^{n}\langle P_h \Wj,\Xbe^{j-1}\rangle_{-1,h}\Big\}^p
\\ & \quad
\leq C \EE\Big\{
T^p+|X^0_h|_{-1,h}^{2p}
\\ & \qquad 
+\Big(\sum_{j=1}^N|P_h\Wj|_{-1,h}^2\Big)^p 
+ \sup_{1\leq n\leq N} \Big|\sum_{j=1}^{n}
\langle P_h\Wj,\Xbe^{j-1}\rangle_{-1,h}\Big|^p
\Big\}
\\ & \quad
\leq C\Big\{  
T^p+|X^0_h|_{-1,h}^{2p}
\\ & \qquad 
+\EE\Big(\sum_{j=1}^N|P_h\Wj|_{-1,h}^2\Big)^p 
+\EE\sup_{1\leq n\leq N} \Big|\sum_{j=1}^{n}\langle P_h
\Wj,\Xbe^{j-1}\rangle_{-1,h}\Big|^p 
\Big\}.
\end{aligned}
\end{equation}

By \eqref{eq:es1a}, we have 
\begin{align}
\label{eq:EW}
\begin{aligned}
&\EE\Big(\sum_{j=1}^N|P_h\Wj|_{-1,h}^2\Big)^p 
\le C N^{p-1}\sum_{j=1}^N \EE|P_h\Wj|_{-1,h}^{2p}
\\ & \quad
\leq C N^{p-1}\sum_{j=1}^N k^p\|A_h^{-1/2}P_hQ^{1/2}\|^{2p}_\HS
\leq C T^p\|A_h^{-1/2}P_hQ^{1/2}\|^{2p}_\HS 
\le C.
\end{aligned}
\end{align}
Moreover, by Cauchy's inequality and \eqref{eq:rbdg}, 
\begin{equation}\label{eq:supest}
\begin{aligned}
&
\Big|\sum_{j=1}^{n}\langle P_h\Wj,\Xbe^{j-1}\rangle_{-1,h}\Big|^p
=\EE\sup_{1\leq n\leq N}\Big|
\sum_{j=1}^{n}\langle \Wj,A_h^{-1}\Xbe^{j-1}\rangle \Big|^p
\\ & \quad
\leq C\EE
\Big(\sum_{j=1}^{N} k\|Q^{1/2}A_h^{-1}\Xbe^{j-1}\|^2\Big)^{p/2}
\\ & \quad
\leq C\Big\{  1+\EE
\Big(\sum_{j=1}^{N} k\|Q^{1/2}A_h^{-1}\Xbe^{j-1}\|^2\Big)^{p}\Big\}
\\ &\quad
\leq C\Big\{1+(Nk)^{p-1}\sum_{j=1}^{N}
k\EE\|Q^{1/2}A_h^{-1}P_h\Xbe^{j-1}\|^{2p} \Big\}
\\ & \quad
\leq C\Big\{1+(Nk)^{p-1}\sum_{j=1}^{N}
k\|Q^{1/2}A_h^{-1/2}P_h\|^{2p}\EE|\Xbe^{j-1}|_{-1,h}^{2p}\Big\}
\\ &\quad
=C\Big\{ 1+(Nk)^{p-1}\sum_{j=1}^{N} k\|A_h^{-1/2}P_hQ^{1/2}\|^{2p}
\EE|\Xbe^{j-1}|_{-1,h}^{2p}\Big\}.
\end{aligned}
\end{equation}
As $Nk=T$,
$\|A_h^{-1/2}P_hQ^{1/2}\|\le \|A_h^{-1/2}P_hQ^{1/2}\|_{\HS}\le K$, and
$|X^0_h|_{-1,h}\le L$, by inserting \eqref{eq:EW} and
\eqref{eq:supest} into \eqref{eq:normest}, we see that
\begin{equation}\label{eq:HmomentGron}
\begin{aligned}
&\EE \sup_{0\leq n\leq N}|\Xben|_{-1,h}^{2p}
\leq C(p,T,K,L)
\Big( 1  +\sum_{j=0}^{N-1} k\EE|\Xbe^{j}|_{-1,h}^{2p} \Big) 
\\ & \quad 
= C(p,T,K,L) 
\Big( 1+\sum_{n=0}^{N-1} k\EE\sup_{0\leq j\leq n}|\Xbej|_{-1,h}^{2p}\Big) .
\end{aligned}
\end{equation}
By induction, \eqref{eq:HmomentGron} shows that the quantity
$\EE\sup_{0\leq j\leq n}|\Xbej|_{-1,h}^{2p}$ is finite for all
$n=1,\dots, N$ and hence the inequality \eqref{eq:min1} follows from
Gronwall's lemma.  Having this result at hand one may return to
\eqref{eq:normsum} to prove \eqref{eq:min2} by a similar procedure but
without the Gronwall argument at the end.
\end{proof}

In the sequel, in many places the quantity
\begin{equation}\label{eq:chem}
Y_h^j:=A_h\Xbej+P_hf(\Xbej)
\end{equation}
plays a crucial role. It can be regarded as the discrete version of
the ``chemical potential'' $Y=AX+f$.  With this notation, the scheme
\eqref{eq:BE} can be rewritten as
\begin{equation}\label{eq:BE1}
\Xbej-\Xbe^{j-1}+kA_hY_h^j
=P_h\Wj,\ j=1,2,\ldots,N; \quad
 X^0_h=P_hX_0.
\end{equation}

We continue to prove a stronger moment bound.  Note that it is not
``closed'', because $Y_h^j$ remains on the right hand side. 

\begin{lem}\label{lem:qest}
  Suppose that $\|A^{1/2}Q^{1/2}\|_{\HS}\le K$ and $|X^0_h|_{-1,h}\le L$ for all $h>0$.
  Then, for every $\epsilon,\delta>0$ and $p\geq 1$, there
  are $C_1>0$ depending on $T,\epsilon,\delta,p,K$ and $L$, and 
  $C_2>0$ depending on $T$ and $p$ such that for all $h,k>0$,
\begin{align}\label{eq:qest}
\begin{aligned}
&\EE\Big(\sup_{1\leq j\leq N}\|\Xbej\|^{2p}\Big)
+\EE\Big(\sum_{j=1}^N\|\Xbej-\Xbe^{j-1}\|^2\Big)^p
\le C_1+C_2\delta\EE \Big(
\sum_{j=1}^N k|Y_h^j|_{1}^2
\Big)^{\frac{1+\epsilon}{2}p}.
\end{aligned}
\end{align}
\end{lem}

\begin{proof}
By taking inner products with $\Xbej\in \dot{S}_h$ in \eqref{eq:BE1} we get
\begin{equation*}
\frac12\Big(
\|\Xbej\|^2-\|\Xbe^{j-1}\|^2+\|\Xbej-\Xbe^{j-1}\|^2\Big)
+k\langle Y_h^j,\Xbej\rangle_1
=\langle\Wj,\Xbej\rangle,
\end{equation*}
where we recall that
  $\langle x,y\rangle_1=\langle\nabla x,\nabla y\rangle$.  Summing
with respect to $j$, using analogous arguments as in the previous
proof, thus yields for $1\le n\le N$
\begin{align*}
\|\Xbe^n\|^2+\sum_{j=1}^n\|\Xbej-\Xbe^{j-1}\|^2 
&\le C\Big(
\|\Xbe^0\|^{2}
+\sum_{j=1}^nk\big|\langle Y_h^j,\Xbej\rangle_1\big|
\\ & \quad 
+\sum_{j=1}^n\|\Wj\|^2 
+\sum_{j=1}^n\langle\Wj,\Xbe^{j-1}\rangle 
\Big).
\end{align*}
Therefore, 
\begin{align*}
&\EE\sup_{1\leq j\leq N}\|\Xbe^j\|^{2p}  
+ \EE\Big(\sum_{j=1}^N\|\Xbej-\Xbe^{j-1}\|^2\Big)^p 
\\ & \quad
\leq C\Big\{
\|X_h^0\|^{2p}
+\EE\Big(\sum_{j=1}^Nk\big|\langle Y_h^j,\Xbej\rangle_1\big|\Big)^p
\\ & \qquad
+\EE\Big(\sum_{j=1}^{N}\|\Wj\|^2\Big)^p 
+\EE\sup_{1\leq n\leq N}\Big|
\sum_{j=1}^n\langle A_h^{1/2}P_h\Wj,A_h^{-1/2}\Xbe^{j-1}\rangle
\Big|^p
\Big\}.
\end{align*}
The next to last term can be bounded, similarly to \eqref{eq:EW} in
the previous proof, by $CT^p\|Q^{1/2}\|_{\HS}^{2p}$.  Using the
\eqref{eq:min1} and a calculation similar to \eqref{eq:supest} we
obtain
\begin{align*}
\EE\sup_{1\leq n\leq N}\Big|\sum_{j=1}^n\langle
  A_h^{1/2}P_h\Wj,A_h^{-1/2}\Xbe^{j-1}\rangle\Big|^p
&\leq
C\Big(1+T^p\|A_h^{1/2}P_hQ^{1/2}\|^{2p}\Big)
\\ &
\leq C\Big(1+T^p\|A^{1/2}Q^{1/2}\|^{2p}\Big),
\end{align*}
where we also used \eqref{chc2:a2}.  Finally, 
\begin{align*}
&\EE\Big(\sum_{j=1}^Nk\Big|\langle Y_h^j,\Xbej\rangle_1\Big|\Big)^p\leq
\EE\Big(\sum_{j=1}^Nk |Y_h^j|_1 |\Xbej|_1\Big)^p\\
& \quad 
\leq \EE\Big\{
\Big(\sum_{j=1}^Nk |Y_h^j|^2_1\Big)^{p/2}
\Big(\sum_{j=1}^Nk|\Xbej |^2_1\Big)^{p/2}
\Big\}
\\ & \quad
\leq \EE\Big\{
\delta\Big( \sum_{j=1}^Nk |Y_h^j|^2_1\Big)^{\frac{1+\epsilon}{2}p}
+C_{\epsilon,\delta}\Big(\sum_{j=1}^Nk|\Xbej|^2_1\Big)^{\frac{1+\epsilon}{2\epsilon}p}
\Big\},
\end{align*}
and the proof is complete in view of \eqref{eq:min2} by noting that
\begin{align*}
\|A_h^{-1/2}P_hQ^{1/2}\|_{\HS}&=\|A_h^{-1/2}P_hA^{-1/2}A^{1/2}Q^{1/2}\|_{\HS}\\
&\leq \|A_h^{-1/2}P_hA^{-1/2}P\| \,\|A^{1/2}Q^{1/2}\|_{\HS}\leq C \|A^{1/2}Q^{1/2}\|_{\HS}.
\end{align*}
\end{proof}

We next prove the main stability result of the paper.

It is well known that,
for the deterministic Cahn--Hilliard equation, 
\begin{align*}
\dot{u}+Av=0, \ t>0; \quad v=Au+f(u),  
\end{align*}
the functional 
\begin{align*}
J(u):=\tfrac12|u|_1^2+\cF(u), \quad\text{where $\cF(u)=\int_{\cD}F(u)\,\dd x$,}
\end{align*}
is a Ljapunov functional, that is, $J(u(t))\le J(u_0)$, $t\ge0$.  This
leads to a uniform bound for
$\snorm[1]{u(t)}^2+\norm[L_4]{u(t)}^4$. The proof proceeds by
multiplication of the equation by the chemical potential $v$ and
noting that $J'(u)=Au+f(u)=v$ and
$\inner{\dot{u}}{v}=\inner{\dot{u}}{J'(u)}=\frac{\dd}{\dd t}J(u)$.
Thus, $\frac{\dd}{\dd t}J(u)+\snorm[1]{v}^2=0$ and 
\begin{align*}
  J(u(T))+\int_0^T\snorm[1]{v(t)}^2\,\dd t=J(u_0), 
\end{align*}
which is the desired result.

This was imitated for the spatially semidiscrete Cahn--Hilliard--Cook
equation in \cite{KLM} by applying the It\^o formula to $J(X_h(t))$.
For the fully discrete equation \eqref{eq:BE} we do not have an It\^o
formula, so we must use a more direct imitation of the above
calculation in the proof of the following theorem, which contains our
main moment bounds.  

In the proof we denote by $P_\alpha(x)$,
$x=(x_1,...,x_m) $, $x_i\geq 0$, any nonnegative quantity such that
\begin{align}
  \label{eq:palfa}
P_\alpha(x)\leq C\Big(1+\sum_{i=1}^mx_i^{\alpha}\Big).  
\end{align}

\begin{thm} \label{thm:H1moment}
Let $p\ge 1$. If $\HSAQ\le K$ and 
\begin{equation}\label{eq:maincond}
|X_h^0|_{-1,h}
+J(X_h^0)+|Y_h^0|_1
\le L\text{ for all }h>0,
\end{equation}
then there exist $C,k_0>0$, depending on $p$, $K$, $L$, and $T$, such
that for all $h>0$ and $0<k<k_0$,
\begin{align*}
\EE\sup_{1\leq j\leq N}J(\Xbej)^{p}
+\EE\Big(\sum_{j=1}^Nk| Y_h^j|_1^2\Big)^p\leq C.
 \end{align*}
\end{thm}

\begin{proof}
  Following the procedure from the deterministic case, we multiply
  \eqref{eq:BE1} by the discrete chemical potential
  $Y_h^j=A_h\Xbej+P_hf(\Xbej)$:
\begin{align*}
\inner{Y_h^j}{\Delta \Xbej}
+k \snorm[1]{Y_h^j}^1 
=\inner{Y_h^j}{P_h\Wj}, 
\end{align*}
where $\Delta \Xbej=\Xbej-\Xbe^{j-1}$.  
From \eqref{eq:taylorf} it follows that, for $X,Z\in \dot{H}$, 
\begin{align*}
\cF(X)-\cF(Z) 
\leq \langle f(X),X-Z\rangle+\tfrac12\betac^2\|X-Z\|^2.
\end{align*} 
Hence, 
\begin{align*}
  \langle P_hf(\Xbej),\Delta \Xbej\rangle 
=  \langle f(\Xbej),\Delta \Xbej\rangle 
\ge \cF(\Xbej)-\cF(\Xbe^{j-1}) 
-\tfrac12\betac^2\|\Delta \Xbej\|^2.
\end{align*}
As in \eqref{eq:betrick} we have 
\begin{align*}
\inner{A_h\Xbej}{\Delta \Xbej}
=\inner[1]{\Xbej}{\Delta \Xbej}
=\frac12 \big(|\Xbej|^2_1-|\Xbe^{j-1}|^2_1
+ |\Delta \Xbej|_1^2\big).
\end{align*}
By adding the latter two relations, we obtain
\begin{align*}
  \inner{Y_h^j}{\Delta \Xbej}
&=\inner{A_h\Xbej+P_hf(\Xbej)}{\Delta \Xbej}
\\ &
\ge 
J(\Xbej)-J(\Xbe^{j-1})+\frac12 |\Delta \Xbej|_1^2
+\tfrac12\betac^2\|\Delta \Xbej\|^2.
\end{align*}
This is the discrete analog of 
$\inner{v}{\dot{u}}=\frac{\dd}{\dd t}J(u)$. 
By \eqref{eq:BE1}, we now have 
\begin{align}\label{eq:Lyapunovest}
J(\Xbej)-J(\Xbe^{j-1})+\frac12 |\Delta \Xbej|_1^2
+k \snorm[1]{Y_h^j}^1 
\le 
\inner{Y_h^j}{P_h\Wj}
-\tfrac12\betac^2\|\Delta \Xbej\|^2.
\end{align}
The remaining challenge is to deal with the term 
\begin{align*}
\inner{Y_h^j}{P_h\Wj}
=
\inner{Y_h^{j-1}}{P_h\Wj}
+\inner{Y_h^j-Y_h^{j-1}}{P_h\Wj}. 
\end{align*}
We begin by 
\begin{align*}
  \inner{Y_h^j-Y_h^{j-1}}{P_h\Wj}
=
\inner{A_h\Delta X_h^j}{P_h\Wj}
+  \inner{f(X_h^j)-f(X_h^{j-1})}{P_h\Wj}.  
\end{align*}
Here, by \eqref{chc2:a2},
\begin{align} \label{eq:Adiffest}
\langle A_h\Delta \Xbej,P_h\Wj\rangle
\leq
\epsilon|\Delta \Xbej|_{1}^2
+C_\epsilon|\Wj|_{1}^2
\end{align}
and, as $P_h\Wj\in \dot{S}_h\subset \dot{H} $,
\begin{equation}\label{eq:Fdiffest}
\begin{aligned}
&|\langle f(\Xbej)-f(\Xbe^{j-1}),P_h\Wj\rangle|
=|\langle P(f(\Xbej)-f(\Xbe^{j-1})),P_h\Wj\rangle|
\\ &  \qquad
=|\langle A_h^{-1/2}P_hP(f(\Xbej)-f(\Xbe^{j-1})),
A_h^{1/2}P_h\Wj\rangle|
\\ & \qquad
\leq \|A_h^{-1/2}P_hP(f(\Xbej)-f(\Xbe^{j-1}))\|\, |P_h\Wj|_{1}
\\ & \qquad
\leq C\|A_h^{-1/2}P_hP(f(\Xbej)-f(\Xbe^{j-1}))\|\,|\Wj|_{1}.
\end{aligned}
\end{equation}
By using H\"older's and Sobolev's inequalities ($d\le3$) we show
\begin{align*}
  \norm{A_h^{-1/2}P_hPf}=\sup_{v_h\in\dot{S}_h}\frac{\inner{f}{v_h}}{\snorm[1]{v_h}}
\le \sup_{v_h\in\dot{S}_h}\frac{\norm[L_{6/5}(\cD)]{f}\norm[L_{6}(\cD)]{v_h}}{\snorm[1]{v_h}}
\le C \norm[L_{6/5}(\cD)]{f}.
\end{align*}
Therefore, \eqref{eq:loclip} implies
\begin{equation}
\begin{aligned}\label{eq:AmFdiff}
&\|A_h^{-1/2}P_hP(f(\Xbej)-f(\Xbe^{j-1})\|
\leq C \|f(\Xbej)-f(\Xbe^{j-1})\|_{L_{6/5}(\cD)}\\
& \qquad 
\leq
C\left(\int_\cD|\Xbej-\Xbe^{j-1}|^{6/5}(1+(\Xbej)^2+(\Xbe^{j-1})^2)^{6/5}\,\dd
  \xi\right)^{5/6}\\
& \qquad \leq C\left(\int_\cD|\Xbej-\Xbe^{j-1}|^{6}\,\dd \xi\right)^{1/6}
\left(\int_\cD(1+(\Xbej)^2+(\Xbe^{j-1})^2)^{3/2}\,\dd \xi\right)^{2/3}\\
& \qquad \leq C\|\Xbej-\Xbe^{j-1}\|_{L_6(\cD)}
\left(1+\|\Xbej\|^2_{L_3(\cD)}+\|\Xbe^{j-1}\|^2_{L_3(\cD)}\right).
\end{aligned}
\end{equation}
Further, with $p<q< r$ and $\lambda =\tfrac pq\tfrac{r-q}{r-p}$, we have
\begin{align*}
\|X\|_{L_{q}(\cD)}\leq \|X\|_{L_{p}(\cD)}^{\lambda}\|X\|_{L_{r}(\cD)}^{1-\lambda},
\end{align*} 
see \cite[Proposition~6.10]{Folland}. 
We take $p=2$, $q=3$, $r=4$, and hence $\lambda=\tfrac13$, to
conclude that
\begin{equation}\label{eq:Finterp}
\|\Xbej\|^2_{L_3(\cD)}\leq \|\Xbej\|^{2/3}_{L_2(\cD)}\|\Xbej\|^{4/3}_{L_{4}(\cD)}.
\end{equation}
Thus, from \eqref{eq:Fdiffest}, \eqref{eq:AmFdiff},
\eqref{eq:Finterp}, and since by Sobolev's inequality we have
$\|\Xbej-\Xbe^{j-1}\|_{L_6(\cD)}\leq C|\Xbej-\Xbe^{j-1}|_{1}$, it
follows that
\begin{equation}\label{eq:AFfinal}
\begin{aligned}
&|\langle f(\Xbej)-f(\Xbe^{j-1}),P_h\Wj\rangle|
\leq
C|\Wj|_1|\Xbej-\Xbe^{j-1}|_{1}
\\ &\qquad
\times
P_{2/3}(\|\Xbej\|_{L_2(\cD)},\|\Xbe^{j-1}\|_{L_2(\cD)})
P_{4/3}(\|\Xbej\|_{L_{4}(\cD)},\|\Xbe^{j-1}\|_{L_{4}(\cD)})
\\ & \quad
\leq \epsilon|\Xbej-\Xbe^{j-1}|_{1}^2
+C|\Wj|_1^2P_{4/3}(\|\Xbej\|_{L_2(\cD)},\|\Xbe^{j-1}\|_{L_2(\cD)})\\
&\qquad 
\times P_{8/3}(\|\Xbej\|_{L_{4}(\cD)},\|\Xbe^{j-1}\|_{L_{4}(\cD)}), 
\end{aligned}
\end{equation}
where we used the notation $P_\alpha$ from \eqref{eq:palfa}.
This means that we have bounded the Lipschitz constant in
\eqref{eq:AmFdiff} by powers of $\norm[L_q(\cD)]{\Xbej}$, $q=2,4$,
which we shall be able to control.  It will be important that the
exponent on $\norm[L_4(\cD)]{\Xbej}$, is strictly less than $4$ so
that it can be controled in terms of $\cF(\Xbej)$.  Therefore we
cannot simplify by multiplying the polynomials together.  The exponent
on $\norm[L_2(\cD)]{\Xbej}$ can be arbitrarily large because of
Lemma~\ref{lem:qest}.

Thus, with $0<\epsilon <\tfrac14$ we get, after inserting
\eqref{eq:Adiffest} and \eqref{eq:AFfinal} into \eqref{eq:Lyapunovest}
and rearranging, that
\begin{align*}
\begin{aligned}
&J(\Xbej)-J(\Xbe^{j-1})+c|\Xbej-\Xbe^{j-1}|_1^2+k|Y_h^j|_1^2\\
&\quad\leq C|\Wj|_1^2P_{4/3}(\|\Xbej\|_{L_2(\cD)},\|\Xbe^{j-1}\|_{L_2(\cD)})
 P_{8/3}(\|\Xbej\|_{L_{4}(\cD)},\|\Xbe^{j-1}\|_{L_{4}(\cD)})
\\ &\qquad\qquad +\langle Y_h^{j-1},\Wj\rangle+C\|\Xbej-\Xbe^{j-1} \|^2.
\end{aligned}
\end{align*}
Summing with respect to $j$ then yields 
\begin{align*}
\begin{aligned}
& J(\Xbe^n)
+\sum_{j=1}^n\Big(|\Xbej-\Xbe^{j-1}|_1^2+k| Y_h^j|_1^2\Big)
\\ &\quad
\leq
C\Big(J(X_h^0)
+\sum_{j=1}^n\|\Xbej-\Xbe^{j-1}\|^2
\\ &\quad\quad 
+\sum_{j=1}^n|\Wj|_1^2
P_{4/3}(\|\Xbej\|_{L_2(\cD)},\|\Xbe^{j-1}\|_{L_2(\cD)})
P_{8/3}(\|\Xbej\|_{L_{4}(\cD)},\|\Xbe^{j-1}\|_{L_{4}(\cD)})
\\ &\quad\quad
+\sum_{j=1}^n\langle Y_h^{j-1},\Wj\rangle \Big).
\end{aligned}
\end{align*}
It follows in a similar way as in the proof of
Lemma~\ref{lem:Hmoment}, using also \eqref{eq:qest} with $\epsilon=1$
and $\delta>0$ so small that the third term on the right hand side
above can be absorbed into the third term in the left hand side below,
that
\begin{equation}\label{eq:anzats}
\begin{aligned}
&\EE\sup_{1\leq n\leq N}J(\Xbe^n)^p 
+\EE\Big(\sum_{j=1}^Nk| Y_h^j|_1^2\Big)^p
\leq C\Big\{1+J(X_h^0)^p
\\ & \qquad 
+N^{p-1}\EE\sum_{j=1}^N
\Big( |\Wj|_1^{2p}
P_{4p/3}(\|\Xbej\|_{L_2(\cD)},\|\Xbe^{j-1}\|_{L_2(\cD)})
\\ & \qquad \quad
\times P_{8p/3}(\|\Xbej\|_{L_{4}(\cD)},\|\Xbe^{j-1}\|_{L_{4}(\cD)}) 
\Big)
+\EE\sup_{1\leq n\leq N}\Big|\sum_{j=1}^n
\langle Y_h^{j-1},\Wj\rangle
\Big|^p 
\Big\}.
\end{aligned}
\end{equation}
The first three terms to the right of the inequality are bounded by
assumption.  For the fourth term we use that
$\EE|\Wj|_1^{2p}\leq C k^p\|A^{1/2}Q^{1/2}\|^{2p}_\HS$, see
\eqref{eq:EW}, together with H\"older's inequality  with
conjugate exponents $q_1,q_1'>1$, to get
\begin{align*}
\begin{aligned}
&N^{p-1}\EE\sum_{j=1}^N\big( |\Wj|_1^{2p}
P_{4p/3}(\|\Xbej\|_{L_2(\cD)},\|\Xbe^{j-1}\|_{L_2(\cD)})
\\ & \qquad\qquad \qquad\qquad
\times P_{8p/3}(\|\Xbej\|_{L_{4}(\cD)},\|\Xbe^{j-1}\|_{L_{4}(\cD)})\big)
\\ & \quad
\leq N^{p-1}\sum_{j=1}^N \big(\EE|\Wj|_1^{2pq_1'} \big)^{1/q_1'}
\Big[\EE\big(P_{4pq_1/3}(\|\Xbej\|_{L_2(\cD)},\|\Xbe^{j-1}\|_{L_2(\cD)})
\\ & \qquad\qquad \qquad\qquad
\times P_{8pq_1/3}(\|\Xbej\|_{L_{4}(\cD)},
\|\Xbe^{j-1}\|_{L_{4}(\cD)})\big)\Big]^{1/q_1}
\\ & \quad
\le  C\HSAQ^{2p} k^{p-1}N^{p-1} \sum_{j=1}^N k
\Big[\EE\big(P_{4pq_1/3}(\|\Xbej\|_{L_2(\cD)},\|\Xbe^{j-1}\|_{L_2(\cD)})
\\ &\qquad\qquad \qquad\qquad
\times P_{8pq_1/3}(\|\Xbej\|_{L_{4}(\cD)},
\|\Xbe^{j-1}\|_{L_{4}(\cD)})\big)\Big]^{1/q_1}.
\end{aligned}
\end{align*}
Here $\HSAQ^{2p}\le C$. H\"older's inequality, now with $q_2,q_2'>1$,
bounds the above quantity by 
\begin{equation}\label{eq:diffisum}
\begin{aligned}
&\leq CT^{p-1}\sum_{j=1}^N k\Big[
\Big(\EE\big(P_{4pq_1q_2'/3}(\|\Xbej\|_{L_2(\cD)},
\|\Xbe^{j-1}\|_{L_2(\cD)})\big)\Big)^{1/q_2'}
\\ &\qquad\qquad \qquad\qquad
\times\Big(\EE\big(P_{8pq_1q_2/3}(\|\Xbej\|_{L_{4}(\cD)}, 
\|\Xbe^{j-1}\|_{L_{4}(\cD)})\big)\Big)^{1/q_2}\Big]^{1/q_1}
\\ &
\leq C\Big\{1+\sum_{j=1}^Nk\Big(
\EE\big(P_{4pq_1q_2'/3}(\|\Xbej\|_{L_2(\cD)},
\|\Xbe^{j-1}\|_{L_2(\cD)})\big)\Big)^{1/q_2'}
\\ & \qquad\qquad \qquad\qquad
\times\Big(\EE\big(P_{8pq_1q_2/3}(\|\Xbej\|_{L_{4}(\cD)}, 
\|\Xbe^{j-1}\|_{L_{4}(\cD)})\big)\Big)^{1/q_2}\Big\}
\\ &
\leq C\Big\{1+\sum_{j=1}^Nk\Big[
\EE\big(P_{4pq_1q_2'/3}(\|\Xbej\|_{L_2(\cD)},
\|\Xbe^{j-1}\|_{L_2(\cD)})\big)
\\ & \qquad\qquad \qquad\qquad
+\EE\big(P_{8pq_1q_2/3}(\|\Xbej\|_{L_{4}(\cD)}, 
\|\Xbe^{j-1}\|_{L_{4}(\cD)})\big)\Big]\Big\}.
\end{aligned}
\end{equation}
With $q_1q_2=3/2$ we use \eqref{eq:F} to get
\begin{align}\label{eq:PF}
\begin{aligned}
P_{8pq_1q_2/3}(\|\Xbej\|_{L_{4}(\cD)}, \|\Xbe^{j-1}\|_{L_{4}(\cD)})
&
\le C(\|\Xbej\|_{L_{4}(\cD)}^{4p}+\|\Xbe^{j-1}\|_{L_{4}(\cD)}^{4p}+1)
\\ &
\leq C(\cF(\Xbej)^{p}+\cF(\Xbe^{j-1})^{p}+1).
\end{aligned}
\end{align}
Furthermore,
\begin{equation}
\begin{aligned}\label{eq:PP}
&\sum_{j=1}^Nk\Big[\EE\big(P_{4pq_1q_2'/3}(\|\Xbej\|_{L_2(\cD)},\|\Xbe^{j-1}\|_{L_2(\cD)})\\ 
&\qquad 
\leq C\EE k \sum_{j=1}^N
(\|\Xbej\|_{L_2(\cD)}^{4pq_1q_2'/3}+\|\Xbe^{j-1}\|_{L_2(\cD)}^{4pq_1q_2'/3}+1)\\ 
& \qquad
=CT+ C\|\Xbe^{0}\|_{L_2(\cD)}^{q}+C\EE k \sum_{j=1}^N \|\Xbej\|_{L_2(\cD)}^{q},
\end{aligned}
\end{equation}
where for brevity $q=4pq_1q_2'/3$.  Let $t,s>1$ be conjugate
exponents. Then, by H\"older's and Young's inequalities,
\begin{equation*}
\begin{aligned}
&
\EE k
\sum_{j=1}^N \|\Xbej\|_{L_2(\cD)}^{q}
\leq \EE \left[\sup_{1\leq j \leq N}\left(
    \|\Xbej\|_{L_2(\cD)}^{q-2}\right)k \sum_{j=1}^N
  \|\Xbej\|_{L_2(\cD)}^{2}\right]
\\
& \quad 
\leq \EE \left[\sup_{1\leq j \leq N}
  \|\Xbej\|_{L_2(\cD)}^{t(q-2)}\right]^{1/t}
\,\left[\EE \left(k \sum_{j=1}^N \|\Xbej\|_{L_2(\cD)}^{2}\right)^s\right]^{1/s}\\
&\quad 
\leq C\left( \EE \left[\sup_{1\leq j \leq N}
    \|\Xbej\|_{L_2(\cD)}^{t(q-2)}\right]
+\EE \left(k \sum_{j=1}^N \|\Xbej\|_{L_2(\cD)}^{2}\right)^s\right)\\
&\quad \leq C\left( \EE \left[\sup_{1\leq j \leq N}
    \|\Xbej\|_{L_2(\cD)}^{t(q-2)}\right]+\EE \left(k \sum_{j=1}^N
    |\Xbej|_1^{2}\right)^s\right). 
\end{aligned}
\end{equation*}
Next, \eqref{eq:qest} from Lemma \ref{lem:qest} and \eqref{eq:min2}
from Lemma \ref{lem:Hmoment} implies that
\begin{equation}\label{eq:sumxest}
\EE k \sum_{j=1}^N \|\Xbej\|_{L_2(\cD)}^{q}\leq C
+K\delta \EE \left(\sum_{j=1}^Nk|Y_h^j|_{1}^2\right)^{\frac{t(q-2)(1+\epsilon)}{4}}.
\end{equation}
We will find $\epsilon>0$, $q_1,q_2, q_2', t>1$, such that $q_2, q_2'$
are conjugate, $q_1q_2=3/2$ and
\begin{equation}\label{eq:condition}
\frac{t(4pq_1q_2'/3-2)(1+\epsilon)}{4}\leq p.
\end{equation}
Since $q_2, q_2'$ are conjugate and $q_1q_2=3/2$ we have that
$q_1q_2'=3/2(q_2'-1)$ and hence \eqref{eq:condition} becomes
\begin{align*}
tp(q_2'-1)\frac{1+\epsilon}{2}-\frac{1+\epsilon}{2}\leq p  
\end{align*}

Note that $q_2<3/2$ and hence $q_2'>3$. If we set
$q_2'=3+\frac{1}{p^2}$ (and thus $q_2=\frac{3p^2+1}{2p^2+1}$ and
$q_1=\frac{6p^2+3}{6p^2+2}$) we need to find $\epsilon>0$ and $t>1$
such that
\begin{equation}\label{eq:tp}
tp(1+\epsilon)+t\frac{1+\epsilon}{2p}-\frac{1+\epsilon}{2}\leq p.
\end{equation}
But
\begin{align*}
tp(1+\epsilon)+t\frac{1+\epsilon}{2p}-\frac{1+\epsilon}{2}\to 
p+\frac{1}{2p}-\frac12<p 
\end{align*}
as $\epsilon\to 0+$ and $t\to 1+$ and hence there is $\epsilon>0$ and
$t>1$ such that \eqref{eq:tp} and hence \eqref{eq:condition}
holds. Therefore, we can conclude from \eqref{eq:PP} and
\eqref{eq:sumxest} that there is $\epsilon>0$, $q_1,q_2, q_2', t>1$,
such that $q_2, q_2'$ are conjugate, $q_1q_2=3/2$ and
\begin{equation}\label{eq:yay}
\begin{aligned}
&\sum_{j=1}^Nk\Big[\EE\big(P_{4pq_1q_2'/3}(\|\Xbej\|_{L_2(\cD)},\|\Xbe^{j-1}\|_{L_2(\cD)})\leq
CT+ C\|\Xbe^{0}\|_{L_2(\cD)}^{4p+\frac{2}{p^2}}\\ 
&\quad+K\delta\left(1+\EE \left(\sum_{j=1}^Nk|Y_h^j|_{1}^2\right)^p\right)
\end{aligned}
\end{equation}
Thus inserting \eqref{eq:PF} and \eqref{eq:yay} into
\eqref{eq:diffisum} we get, with $C=C(T,p)>0$ and $K=K(T,p)>0$ independent of $\delta>0$,
\begin{equation}
\begin{aligned}\label{eq:term4}
&N^{p-1}\EE\sum_{j=1}^N |\Wj|_1^{2p} P_{4p/3}(\|\Xbej\|_{L_2(\cD)})
 P_{4p/3}(\|\Xbej\|_{L_{4}(\cD)})\\
 & \quad 
\leq
C\left(1+\|\Xbe^0\|_{L_2(\cD)}^{4p+\frac{2}{p^2}}+\sum_{j=0}^Nk\EE\cF(\Xbej)^{p}\right)+
K\delta\left(1+\EE \left(\sum_{j=1}^Nk|Y_h^j|_{1}^2\right)^p\right). 
\end{aligned}
\end{equation}

It remains to treat the It\^o integral in \eqref{eq:anzats}. For this
we invoke the Cauchy inequality and Burkholder--Davis--Gundy
inequality to conclude that
\begin{equation}\label{eq:term5}
\begin{aligned}
&\EE\sup_{1\leq n\leq N}\Big|\sum_{j=1}^n\langle Y_h^{j-1},\Wj\rangle\Big|^p
\\ & \quad
\leq C\Big(1 +\epsilon'\EE\sup_{0\leq n\leq N}\Big|\sum_{j=1}^n
\langle Y_h^{j-1},\Wj\rangle\Big|^{2p}\Big)
\\ & \quad 
\leq C\Big(1 +\epsilon'\EE\Big(\sum_{j=1}^Nk
\|Q^{1/2}Y_h^{j-1}\|^2
\Big)^{p}\Big)
\\ &\quad 
\leq C\Big(1+\epsilon'\EE\Big(
\sum_{j=1}^Nk\|Q^{1/2}A^{-1/2}P\| 
| Y_h^{j-1}|_1^2\Big)^{p}\Big).
\end{aligned}
\end{equation}
Thus, since $\|Q^{1/2}A^{-1/2}P\|<\infty$, if we take $\epsilon'>0$ small
enough in \eqref{eq:term5} and  $\delta>0$ small enough in \eqref{eq:term4},  
from \eqref{eq:anzats} we may conclude that
\begin{equation}
\begin{aligned}
&\EE\sup_{1\leq n\leq N}J(\Xbe^n)^p+\EE\Big(\sum_{j=1}^Nk| Y_h^j|_1^2\Big)^p\\
&\quad \leq
C\Big(1+J(X_h^0)^p+\|\Xbe^0\|_{L_2(\cD)}^{4p+\frac{1}{p^2}}+k\cF(\Xbe^{0})^{p}+k|Y_h^0|_1^{2p}\\ 
&\qquad+\sum_{j=1}^Nk\EE\sup_{1\leq n\leq j}\cF(\Xbe^{n})^{p}\Big).
\end{aligned}
\end{equation}
Thus, if $Ck<1 $, then the desired result follows from Gronwall's
lemma by noting that $\cF(u)\leq J(u)$.
\end{proof}

\section{Convergence}\label{sec:conv}
Recall from Theorem~\ref{thm:existenceandregularityofX} that $X$
satisfies the mild equation
\begin{align} \label{eq:femtielva}
X(t)=E(t)X_0-\int_0^t AE(t-s)f(X(s))\,\dd s+\int_0^t E(t-s)\,\dd W(s).
\end{align}
Similarly, equation \eqref{eq:BE} has the mild formulation
\begin{align} \label{eq:femtitolv}
\Xbe^n=\Rk^nP_hX_0-\sum_{j=1}^nA_h\Rk^{n-j+1}P_hf(\Xbej)+\sum_{j=1}^n\Rk^{n-j+1}P_h\Wj,
\end{align}
where $\Rk^n=(I+k\, A_h^2)^{-n}$.  

\begin{rem} \label{rem:preservemass} Preservation of mass. From
    \eqref{chsg}, \eqref{eq:femtielva} and
    \eqref{chsg2},\eqref{eq:femtitolv} it follows that if $W(t)$ has
    zero average, i.e., $(I-P)W(t)=0$, then $(I-P)X(t)=(I-P)X_0$ and
    $(I-P)\Xbe^n=(I-P)X_0$. This means that $X(t)$ and $\Xbe^n$ preserve
    the mass.
\end{rem}

In order to prove convergence of $\Xbe^n$, we first state a maximal
type error estimate for the stochastic convolution.  We define the
backward Euler approximation of the stochastic convolution
$W_A(t):=\int_0^t E(t-s)\,\dd W(s)$ by
\begin{align*}
W_{A_h}^n:=\sum_{j=1}^n\Rk^{n-j+1}P_h\Delta W^j
=\sum_{j=1}^n\int_{t_{j-1}}^{t_j}\Rk^{n-j+1}P_h\,\dd W(s).
\end{align*}

\begin{lem}\label{lem:wa}
Let $\gamma\in(0,\frac12]$, $\beta \in[1,2]$, and
  $p\ge 1$.  Then there is $C=C(p,\gamma,T)$ such that for $h,k>0$
  \begin{align}\label{eq:stco}
  \left(\EE\Big(\sup_{0\le n\le N}\|W_A(t_n)-W^n_{A_h}\|^p\Big)\right)^{1/p}
  \le C(h^\beta+\ts^{\beta/4 })\|A^{(\beta-2)/2+\gamma} PQ^{1/2}\|_{\HS}.
  \end{align}
\end{lem}
\begin{proof}
  The proof is completely analogous to the proofs of
  \cite[Proposition~5.1]{KLLAC} and \cite[Theorem~2.1]{MR3273327},
  based on a discrete factorization method using the analyticity of
  the semigroup $E$ and the deterministic error estimate
\begin{equation}
\label{eq:lchce}
\|(E(t_n)-\Rk^n)P_hv\|
\le C(h^{\beta} +\ts^{\beta/4} )t_{n}^{-(\beta-\alpha)/{4}}|v|_{\alpha},
\quad t_n>0,
\end{equation}
for $\beta\in[1,2]$, $\alpha\in[-1,1]$, from \cite[Lemma 5.5]{Stig}
and \cite[Theorem 2.2]{MR2846757}.
\end{proof}
\begin{rem} In \cite[Proposition~5.1]{KLLAC}, instead of the
    term $\|A^{(\beta-2)/2+\gamma} PQ^{1/2}\|_{\HS}$ which appears on
    the right hand side of \eqref{eq:stco}, one has
    $\|A^{(\beta-1)/2+\gamma}Q^{1/2}\|_{\HS}$. The reason for this
    difference (besides the different boundary conditions) is that
    \cite{KLLAC} considers the stochastic Allen--Cahn equation where
    the semigroup in the stochastic convolution is generated
    by the Laplacian $\Delta$ which has a weaker smoothing effect than
    in the present case, where the semigroup is generated by
    $-\Delta^2$. The proofs in \cite{KLLAC} and \cite{MR3273327}
    require that $p$ is large, but the result is then valid for
    smaller $p\ge1$ as well.
\end{rem}
\begin{lem}\label{lem:dete}
Let $0<\delta<1$.  The following deterministic error estimates hold for $v\in H$:
\begin{align}
\|A_h\Rk^{n}P_hv-A_hE_h(t_n)P_hv\|
&\leq Ck^{\frac12(1-\delta)}t_n^{-1+\frac{\delta}{2}}\|v\|,\quad t_n>0\text{ and }h,k>0,
\label{eq:dette}
\\
\|A_hE_h(t)P_hv-AE(t)v\|
&\leq Ch^{2(1-\delta)}t^{-1+\frac{\delta}{2}}\|v\|, \quad t>0\text{ and }h>0.
\label{eq:dette1} 
\end{align}
\end{lem}

\begin{proof}
Note that it is enough to consider $v\in \dot{H}$, as for $v$
constant the above differences equal 0.  The error bounds follow by a
simple interpolation between the cases $\delta=0$ and $\delta=1$.  For
$\delta=1$ we use the estimates from \eqref{eq:analytic} and \eqref{eq:fem2} to get
\begin{equation*}
\|A_h\Rk^{n}P_h\|+\|A_hE_h(t_n)P_h\|\leq C t_n^{-1/2},\quad 
\|A_hE_h(t)P_h\|+\|AE(t)\|\leq Ct^{-1/2}. 
\end{equation*}
Estimate \eqref{eq:dette} with $\delta =0$ follows by expansion
  in the eigenbasis of $A_h$ and Parseval's identity. For estimate
\eqref{eq:dette1} with $\delta=0$, we first write, for $v\in \dot{H}$,
\begin{align*}
&\|A_hE_h(t)P_hv-AE(t)v\|
\leq 
\|(A_h^2E_h(t)P_h-A^2E(t))A^{-1}v\|
\\ & \qquad \quad
+\|A_h^2E_h(t)P_h(A_h^{-1}P_h-A^{-1})v\|
\\ & \qquad 
=\|D_t(E_h(t)P_h-E(t))A^{-1}v\|
+\|A_h^2E_h(t)P_h(R_h-I)A^{-1}v\|.
\end{align*}
where we used the identity $R_h=A_h^{-1}P_hA$. The desired
bound $Ct^{-1}h^2\norm{A(A^{-1}v)}=Ch^2t^{-1}\norm{v}$ for the last term follows
immediately from \eqref{eq:fem1} and \eqref{eq:fem2}.  The first term
is an error estimate for the time derivative of the solution of the
linear Cahn--Hilliard equation with smooth initial-value
$u_0=A^{-1}v\in\dot{H}^2$.  To prove this we adapt the arguments in
\cite[Chapt.~3]{Thomeebook} and \cite[Sect.~5]{Stig}, where error
estimates with lower initial regularity are proved.  Let
$u(t)=E(t)u_0$, $u_h(t)=E_h(t)P_hu_0$.  Then the error $e=u_h-u$
satisfies the equation, see \cite[(5.4)]{Stig},
\begin{align}
  \label{eq:error1}
G_h^2\dot{e}+e=\rho+G_h\eta,\ t>0;
\quad P_he(0)=0,
\end{align}
with 
\begin{align}
  \label{eq:error2}
  G_h=A_h^{-1}P_h, \quad R_h=G_hA, \quad \rho=(R_h-I)u, \quad \eta=-(R_h-I)A^{-1}\dot{u}.
\end{align}
Any solution of an equation of the form \eqref{eq:error1} satisfies
the following bound, with arbitrary $\epsilon>0$,
\begin{align}
  \label{eq:error3}
  \|e(t)\|\le \epsilon \sup_{s\in[0,t]}s\|\dot{\rho}(s)\| 
  +C_\epsilon \sup_{s\in[0,t]}\|\rho(s)\| 
+\Big( \int_0^t \|\eta(s)\|^2\,\dd s \Big)^{1/2}.  
\end{align}
In order to prove this we let $e_1$ be the solution of
\eqref{eq:error1} with only $\rho$ as the source term and
$P_he_1(0)=0$.  Moreover, we let $e_2$ solve the same equation but
driven by $G_h\eta$ alone and with $e_2(0)=0$.  Then $e=e_1+e_2$
solves \eqref{eq:error1}.  We quote a bound for $e_1$ from
\cite[Lemma~3.5]{Thomeebook}:
\begin{align*}
  \|e_1(t)\|\le \epsilon \sup_{s\in[0,t]}s\|\dot{\rho}(s)\| 
  +C_\epsilon \sup_{s\in[0,t]}\|\rho(s)\| . 
\end{align*}
In order to quote this lemma we note that $G_h$ is
selfadjoint, positive semidefinite on $\dot{H}$ and that
$G_he_1(0)=A_h^{-1}P_he_1(0)=0$.  For $e_2$ we have
\begin{align*}
    \|e_2(t)\|\le \Big( \int_0^t \|\eta(s)\|^2\,\dd s \Big)^{1/2}.  
\end{align*}
This is proved by a simple energy argument, see the beginning of the
proof of \cite[Lemma~5.2]{Stig}.  The reason why we need different
proofs for $e_1$ and $e_2$ is that $G_h$ in front of $\eta$ must not
appear in \eqref{eq:error3} for we have good bounds for $\eta$ but not
for $G_h\eta$.  

This proves \eqref{eq:error3}, which can now be combined with bounds
for $\rho$ and $\eta$, obtained from bounds for $R_h$ and regularity
estimates for $u=E(t)u_0$, to get an error bound for $\|e(t)\|$.
However, we aim for  $\|\dot{e}(t)\|$ and therefore take the
derivative of the equation in \eqref{eq:error1} and multiply by $t$ to
obtain an equation for $t\dot{e}(t)$:  
\begin{align*}
  tG_h^2\ddot{e}+t\dot{e} 
=t\dot{\rho}+tG_h\dot{\eta},
\end{align*}
which can be written as 
\begin{align*}
G_h^2(t\dot{e})\dot{}+(t\dot{e}) 
=G_h^2\dot{e}+t\dot{\rho}+G_h(t\dot{\eta} )
=
-e+\rho+t\dot{\rho}+G_h(\eta+t\dot{\eta}) ,
\end{align*}
where we substituted $G_h^2\dot{e}=-e+\rho+G_h\eta$ from
\eqref{eq:error1}.  Thus,  $t\dot{e}$ satisfies an equation of the
form \eqref{eq:error1} but with $\rho$ and $\eta$ replaced by
$-e+\rho+t\dot{\rho}$ and $\eta+t\dot{\eta}$.  An application of 
\eqref{eq:error3} with $\epsilon=\frac12$, say, gives 
\begin{align*}
t\|\dot{e}(t)\|
&
\le \tfrac12 \sup_{s\in[0,t]}\big(
s\|\dot{e}(s)\|+ 2s\|\dot{\rho}(s)\| + s^2\|\ddot{\rho}(s)\|
\big)
\\ & \quad 
+C \sup_{s\in[0,t]}\big( 
\|e(s)\|+\|\rho(s)\| +s\|\dot{\rho}(s)\|
\big)
\\ & \quad 
+\Big( \int_0^t\big( \|\eta(s)\|^2+s^2\|\dot{\eta}(s)\|^2\big)\,\dd s \Big)^{1/2}.  
\end{align*}
Since $t$ is arbitrary here we may apply a standard kick-back argument
to remove the term $s\|\dot{e}(s)\|$.  Another application of
\eqref{eq:error3}, now with $\epsilon=1$, takes care of the term
$\|e(s)\|$, which leads to 
\begin{align*}
t\|\dot{e}(t)\|
&\le C\sup_{s\in[0,t]}\big(
\|\rho(s)\| +s\|\dot{\rho}(s)\| + s^2\|\ddot{\rho}(s)\|
\big)
\\ &\quad 
+C\Big( \int_0^t\big( \|\eta(s)\|^2+s^2\|\dot{\eta}(s)\|^2\big)\,\dd s \Big)^{1/2}.  
\end{align*}
Here we use \eqref{eq:fem1} and recall the regularity estimates
\eqref{eq:analytic}, \eqref{eq:analytic1} for $u(t)=E(t)A^{-1}v$:
\begin{align*}
  s^j\|D_s^j\rho(s)\| 
 = s^j\|(R_h-I)D_s^ju(s)\| 
 \le Ch^2 s^j\|AD_s^j E(s)A^{-1}v\| 
 \le Ch^2 \|v\| 
\end{align*}
and
\begin{align*}
  &\Big( \int_0^t s^{2j}\|D_s^j{\eta}(s)\|^2\,\dd s \Big)^{1/2}
=
  \Big( \int_0^t s^{2j}\|(R_h-I)A^{-1}D_s^j\dot{u}(s)\|^2\,\dd s \Big)^{1/2}
\\ & \qquad 
\le Ch^2  \Big( \int_0^t s^{2j}\|D_s^{j+1}u(s)\|^2\,\dd s \Big)^{1/2}
\\ & \qquad 
\le Ch^2  \Big( \int_0^t s^{2j}\|A^{2j}E(s)A^{-1}v\|^2\,\dd s \Big)^{1/2}
\\ & \qquad 
\le Ch^2  \Big( \int_0^t s^{2j}\|A^{2j+1}E(s)v\|^2\,\dd s \Big)^{1/2}
\le Ch^2 \|v\|.
\end{align*}
This completes the proof. 
\end{proof}

\begin{thm}\label{thm:pconv}
  Suppose that \eqref{eq:maincond} holds,
  $\|A^{1/2}Q^{1/2}\|_{\HS}<\infty$, $\beta\in[1,2]$, and that
  $|X_0|_\beta<\infty$.  Let $h>0$ and $k>0$ be small and
  $0<\epsilon,\delta<1$. Then, there is
  $\Omega^{\epsilon}_{h,k}\subset \Omega$ with
  $\mathbf{P}(\Omega^{\epsilon}_{h,k})>1-\epsilon$, and
  $C=C(T,\epsilon,\delta)$ such that for all
  $\omega\in \Omega^{\epsilon}_{h,k}$,
\begin{align*}
\|X(t_n)-\Xbe^n\|
\leq 
C\left((h^\beta+\ts^{\beta/4} ) |X_0|_{\beta}
+h^{2(1-\delta)}+\ts^{\frac12(1-\delta)}\right),\quad t_n\in [0,T].    
\end{align*}
\end{thm}

\begin{proof}
  It follows from Proposition~\ref{prop:ph} for
    \eqref{eq:plip}, Theorem~\ref{thm:existenceandregularityofX} for
    \eqref{eq:path1}, Theorem~\ref{thm:H1moment} for \eqref{eq:path2},
    and Lemma~\ref{lem:wa} for \eqref{eq:path3} that for every
  $0<\epsilon,\delta<1$ and $h,k>0$ small enough, there is
  $\Omega^{\epsilon}_{h,k}\subset \Omega$ with
  $\mathbf{P}(\Omega^{\epsilon}_{h,k})>1-\epsilon$ and
  $K_{T,\epsilon}>0$ such that
\begin{align}
\|X(t)-X(s)\|
&\leq  K_{T,\epsilon}|t-s|^{\frac12(1-\delta)},
&&\quad t,s\in [0,T],\ \omega\in \Omega^{\epsilon}_{h,k},\label{eq:plip}\\
|X(t)|^2_{1}+\|X(t)\|^4_{L_4}
&\leq  K_{T,\epsilon},
&&\quad t\in [0,T],\ \omega\in \Omega^{\epsilon}_{h,k},\label{eq:path1}\\
|\Xbe^n|^2_1+\|\Xbe^n\|^4_{L_4}
&\leq K_{T,\epsilon},
&&\quad t_n\in [0,T],\ \omega\in \Omega^{\epsilon}_{h,k}\label{eq:path2}\\
\|W_A(t_n)-W^n_{A_h}\|
&\leq K_{T,\epsilon}(h^2+\ts^{1/2}),
&&\quad t_n\in [0,T],\ \omega\in \Omega^{\epsilon}_{h,k},\label{eq:path3}.
\end{align}
Note that it is enough to establish the above four bounds individually
with $\epsilon/4$ on $\Omega^{\epsilon/4,i}_{h,k}$, $i=1,...,4$, and
then set
$\Omega^{\epsilon}_{h,k}=\cap_{i=1}^4 \Omega^{\epsilon/4,i}_{h,k}$.
The estimate in \eqref{eq:plip} follows directly from the assumption
on the initial data and Proposition~\ref{prop:ph}.  The remaining
bounds can be proved by using Chebychev's inequality together with
bounds from Theorem~\ref{thm:existenceandregularityofX},
Theorem~\ref{thm:H1moment}, and Lemma~\ref{lem:wa}. For example, to
prove \eqref{eq:path3}, consider
\begin{align*}
F_{h,k}:=\frac{\sup_{0\le n\le N}\|W_A(t_n)-W^n_{A_h}\|}{h^2+\ts^{1/2}}.  
\end{align*}
Chebychev's inequality and Lemma~\ref{lem:wa}, for some $p\ge1$,
$\gamma=\frac{1}{2}$, and $\beta=2$, give
\begin{equation*}
\mathbf{P}\Big(\big\lbrace \omega \in \Omega:  F_{h,k} > \alpha \big\rbrace \Big)
\le \frac{1}{\alpha^p} \IE\big[F_{h,k}^p\big]
\le \frac{K^p}{\alpha^p}, \quad K=C\norm[\HS]{A^{1/2}Q^{1/2}}.  
\end{equation*}
We choose
$\alpha = \epsilon^{-1/p}K$ and set
$\Omega^\epsilon_{h,k}=\big \lbrace \omega \in \Omega : F_{h,k} \le
\epsilon^{-1/p}K \big\rbrace$.  Then
\begin{equation*}
\mathbf{P}(\Omega^\epsilon_{h,k})
= 1- \mathbf{P}\Big(\big\lbrace \omega \in \Omega: 
F_{h,k} > \epsilon^{-1/p}K  \big\rbrace \Big)
\ge 1-\epsilon,
\end{equation*}
and \eqref{eq:path3} follows.

Now let $\omega \in \Omega^{\epsilon}_{h,k}$.  We decompose the error
$e_n:=X(t_n)-\Xbe^n$ as
\begin{align*}
\begin{aligned}
e_n
&
=\Big(E(t_n)-\Rk^nP_h\Big)X_0
\\ &\quad 
+\sum_{j=1}^n\int_{t_{j-1}}^{t_j}
\Big(A_h\Rk^{n-j+1}P_hf(\Xbej)-AE(t_n-s)f(X(s))\Big)\,\dd s
\\ &\quad 
+W_A(t_n)-W^n_{A_h}=:e_n^1+e_n^2+e_n^3.
\end{aligned}
\end{align*}
For the first error term we use \eqref{eq:lchce} to get
\begin{align*}
\|e_n^1\|\leq C(h^{\beta}+\ts^{\beta/4})|X_0|_{\beta}.  
\end{align*}
The term $e_n^2$ is most involved and we decompose it further as
\begin{align*}
\begin{aligned}
&  A _h\Rk^{n-j+1}P_hf(\Xbej)-AE_h(t_n-s)f(X(s))
\\ &\qquad
=
\Big(A_h\Rk^{n-j+1}-A_hE_h(t_n-t_{j-1})\Big)P_hf(\Xbej)
\\ &\qquad \quad
+\Big(A_hE_h(t_n-t_{j-1})P_h-AE(t_n-t_{j-1})\Big)f(\Xbej)
\\ &\qquad \quad
+ A^{3/2}E(t_n-t_{j-1})A^{-1/2}P\Big(f(\Xbej)-f(X(t_j))\Big) 
\\ &\qquad \quad
+A\Big(E(t_n-t_{j-1})-E(t_n-s)\Big)f(X(t_j))
\\ &\qquad \quad
+A^{3/2}E(t_n-s)A^{-1/2}P\Big(f(X(t_j))-f(X(s))\Big)
\\ & \qquad
=:e_{n,j}^{2,1}+e_{n,j}^{2,2}+e_{n,j}^{2,3}+e_{n,j}^{2,4}+e_{n,j}^{2,5}.
\end{aligned}
\end{align*}
Here we used \eqref{chsg} to obtain $AE(t)=A^{3/2}E(t)A^{-1/2}P$.
Further, since $f$ is cubic we have $\|f(x)\|\le C(1+|x|_1^3)$ by
H\"older's and Sobolev's inequalities. Then by \eqref{eq:dette}, and
\eqref{eq:path2} it follows that, with $C=C(T,\epsilon,\delta)$,
\begin{align*}
\|e_{n,j}^{2,1}\|
&
\leq 
Ck^{\frac12(1-\delta)}t_{n-j+1}^{-1+\frac{\delta}{2}}\|f(\Xbej)\|
\leq 
Ck^{\frac12(1-\delta)}t_{n-j+1}^{-1+\frac{\delta}{2}}\big(1+|\Xbej|^3_{1}\big)
\\ & 
\leq Ck^{\frac12(1-\delta)}t_{n-j+1}^{-1+\frac{\delta}{2}}.
\end{align*}
Similarly, by \eqref{eq:dette1} and \eqref{eq:path2},
\begin{align*}
\|e_{n,j}^{2,2}\|
\leq C h^{2(1-\delta)}t_{n-j+1}^{-1+\frac{\delta}{2}}.
\end{align*}
Also, using the local Lipschitz bound
$|P(f(x)-f(y))|_{-1}\le C (1+|x|^2_1+|y|^2_1)\|x-y\| $, cf.\
\eqref{eq:AmFdiff}, together with \eqref{eq:path1}, and
\eqref{eq:path2}, we obtain
\begin{align*}
\begin{aligned}
\|e_{n,j}^{2,3}\|
&
\leq C(t_n-t_{j-1})^{-3/4}|P(f(\Xbej)-f(X(t_j))|_{-1}
\\ &
\leq Ct_{n-j+1}^{-3/4}(1+|\Xbej|^2_1+|X(t_j)|^2_1)\|\Xbej-X(t_j)\|
\\ &
\leq C t_{n-j+1}^{-3/4}\|e_j\|.
\end{aligned}
\end{align*}
Furthermore, for $s\in [t_{j-1},t_j]$, by \eqref{eq:analytic} and
$\|(E(t)-I) x \|\le Ct^{\frac12(1-\delta)}\|A^{1-\delta}x\|$, we have 
\begin{align*}
\begin{aligned}
\|e_{n,j}^{2,4}\|
&
=
\|(E(s-t_{j-1})-I) A^{-1+\delta} A^{2-\delta}E(t_n-s)f(X(t_j))\|
\\ &
\leq C
(s-t_{j-1})^{\frac12(1-\delta)}(t_n-s)^{-1+\frac{\delta}{2}}\|f(X(t_j))\|
\\ &
\leq C\ts^{\frac12(1-\delta)}(t_n-s)^{-1+\frac{\delta}{2}}.
\end{aligned}
\end{align*}
Using also \eqref{eq:plip}, for $s\in [t_{j-1},t_j]$, we have 
\begin{align*}
\begin{aligned}
\|e_{n,j}^{2,5}\| 
&
\leq C(t_n-s)^{-3/4}
\big(1+|X(t_j)|_1^2+|X(s)|^2_{1}\big)\| X(t_j)-X(s)\|
\\ & 
\leq 
C(t_n-s)^{-3/4}(t_j-s)^{\frac12(1-\delta)}
\leq 
C(t_n-s)^{-3/4}\ts^{\frac12(1-\delta)}. 
\end{aligned}
\end{align*}
Finally, by \eqref{eq:path3}, $\|e_n^3\|\leq C(h^2+\ts^{1/2})$.
Collecting all the above terms and applying a generalized version of
Gronwall's lemma, \cite[Lemma 7.1]{Stig}, finishes the proof if $k$ is
small enough.
\end{proof}

\begin{thm}\label{thm:main}
Under the hypothesis of Theorem \ref{thm:pconv} with $\beta =1$, we
have 
\begin{align*}
\lim_{h,k\to 0}\EE\sup_{0\le n\le N}\|X(t_n)-\Xbe^n\|^2=0.  
\end{align*}
\end{thm}

\begin{proof}
It follows from Theorem~\ref{thm:existenceandregularityofX} and
Theorem~\ref{thm:H1moment} that there is $K>0$ such that
\begin{align*}
  \EE\sup_{0\le n\le N} 
  \left(\|X(t_n)\|_{L_4}^4+\|\Xbe^n\|_{L_4}^4\right)\le K. 
\end{align*}
Let $\epsilon>0$, $0<h,k<1$ small enough, and let
$C_\epsilon=C(T,\epsilon,\delta)$ and $\Omega^{\epsilon}_{h,k}$ as in
Theorem~\ref{thm:pconv}. Then, by using Theorem~\ref{thm:pconv} with
$\beta=1$ and $\delta=\frac12$, we get
\begin{equation*}
\begin{aligned}
&\EE\sup_{0\le n\le N}\|X(t_n)-\Xbe^n\|^2
\leq \int_{\Omega^{\epsilon}_{h,k}}\sup_{0\le n\le N}\|X(t_n)-\Xbe^n\|^2\,\dd \mathbf{P}\\
&\qquad\quad 
+2\int_{(\Omega^{\epsilon}_{h,k})^c}\sup_{0\le n\le N}
\left(\|X(t_n)\|^2+\|\Xbe^n\|^2\right)\,\dd \mathbf{P}\\
&\qquad
\leq C_\epsilon(h^{2}+k^{1/2})
+4 \epsilon^{1/2}
\left(\int_{(\Omega^{\epsilon}_{h,k})^c}\sup_{0\le n\le N}
\left(\|X(t_n)\|^4+\|\Xbe^n\|^4\right)
\,\dd \mathbf{P}\right)^{1/2}\\
&\qquad
\leq C_\epsilon(h^{2}+k^{1/2})
+4 \epsilon^{1/2}
\left(\EE\sup_{0\le n\le N}\left(\|X(t_n)\|^4+\|\Xbe^n\|^4\right)\right)^{1/2}\\
&\qquad
\leq C_\epsilon(h^{2}+k^{1/2})
+4 \epsilon^{1/2}|\mathcal{D}|^{1/2}
\left(\EE\sup_{0\le n\le N}\left(\|X(t_n)\|_{L_4}^4+\|\Xbe^n\|_{L_4}^4\right)\right)^{1/2}
\\ &\qquad
\leq C_\epsilon(h^{2}+k^{1/2})+4 \epsilon^{1/2}|\mathcal{D}|^{1/2}K.
\end{aligned}
\end{equation*}
Let $\eta>0$. Choose $0<\epsilon<1$ such that $8
\epsilon^{1/2}|\mathcal{D}|^{1/2}K<\frac{\eta}{2}$. 
Therefore, if $\max(h,k)<\left(\frac{\eta}{4C_\epsilon}\right)^{2}$, then
\begin{align*}
\EE\sup_{0\le n\le N}\|X(t_n)-\Xbe^n\|^2< \eta,  
\end{align*}
and the proof is complete.
\end{proof}

\subsection*{Acknowledgement}  We thank the anonymous referees for
constructive criticism that helped to clarify the presentation of the
paper.


\begin{thebibliography}{10}

\bibitem{2016arXiv160105756B}
{\sc S.~{Becker} and A.~{Jentzen}}, {\em {Strong convergence rates for
  nonlinearity-truncated Euler-type approximations of stochastic
  Ginzburg-Landau equations}}.
\newblock arXiv:1601.05756, 2016, \url{https://arxiv.org/abs/1601.05756}.

\bibitem{Blomker_Jentzen}
{\sc D.~Bl{\"o}mker and A.~Jentzen}, {\em Galerkin approximations for the
  stochastic {B}urgers equation}, SIAM J. Numer. Anal., 51 (2013),
  pp.~694--715, \url{https://doi.org/10.1137/110845756}.

\bibitem{Blomker_et_al}
{\sc D.~Bl{\"o}mker, M.~Kamrani, and S.~M. Hosseini}, {\em Full discretization
  of the stochastic {B}urgers equation with correlated noise}, IMA J. Numer.
  Anal., 33 (2013), pp.~825--848, \url{https://doi.org/10.1093/imanum/drs035}.

\bibitem{MR1870315}
{\sc D.~Bl{\"o}mker, S.~Maier-Paape, and T.~Wanner}, {\em Spinodal
  decomposition for the stochastic {C}ahn-{H}illiard equation}, in
  International {C}onference on {D}ifferential {E}quations, {V}ol. 1, 2
  ({B}erlin, 1999), World Sci. Publ., River Edge, NJ, 2000, pp.~1265--1267.

\bibitem{MR3081484}
{\sc Z.~Brze{\'z}niak, E.~Carelli, and A.~Prohl}, {\em Finite-element-based
  discretizations of the incompressible {N}avier-{S}tokes equations with
  multiplicative random forcing}, IMA J. Numer. Anal., 33 (2013), pp.~771--824,
  \url{https://doi.org/10.1093/imanum/drs032}.

\bibitem{CW}
{\sc C.~Cardon-Weber}, {\em Implicit approximation scheme for the {C}ahn-{H}illiard
  stochastic equation}.
\newblock Preprint, Laboratoire des Probabilit\'es et Model\`eles Al\'eatoires,
  Universit\'e Paris V, 2000.

\bibitem{MR3022227}
{\sc E.~Carelli and A.~Prohl}, {\em Rates of convergence for discretizations of
  the stochastic incompressible {N}avier-{S}tokes equations}, SIAM J. Numer.
  Anal., 50 (2012), pp.~2467--2496, \url{https://doi.org/10.1137/110845008}.

\bibitem{Cook}
{\sc H.~E. Cook}, {\em Brownian motion in spinodal decomposition}, Acta
  Metallurgica, 18 (1970).




\bibitem{MR1359472}
{\sc G.~Da~Prato and A.~Debussche}, {\em Stochastic {C}ahn-{H}illiard
  equation}, Nonlinear Anal., 26 (1996), pp.~241--263,
  \url{https://doi.org/10.1016/0362-546X(94)00277-O}.

\bibitem{DPZ}
{\sc G.~Da~Prato and J.~Zabczyk}, {\em Stochastic {Equations} in {Infinite}
  {Dimensions}}, vol.~44 of Encyclopedia of Mathematics and its Applications,
  Cambridge University Press, Cambridge, 1992.

\bibitem{Dav}
{ \sc E.~B. Davies},  {\em Spectral {Theory} and {Differential} {Operators}},
Cambridge Studies in Advanced Mathematics, 42, Cambridge University Press, Cambridge, 1995.

\bibitem{Stig}
{\sc C.~M. Elliott and S.~Larsson}, {\em Error estimates with smooth and
  nonsmooth data for a finite element method for the {C}ahn-{H}illiard
  equation}, Math. Comp., 58 (1992), pp.~603--630, S33--S36,
  \url{https://doi.org/10.2307/2153205}.

\bibitem{Folland}
{\sc G.~B. Folland}, {\em Real {Analysis}}, John Wiley \& Sons, Inc., New York,
  1999.

%
%
\bibitem{GyM}
{\sc I.~Gy{\"o}ngy and A.~Millet}, {\em On discretization schemes for
  stochastic evolution equations}, Potential Anal., 23 (2005), pp.~99--134,
  \url{https://doi.org/10.1007/s11118-004-5393-6}.


\bibitem{MR3498982}
{\sc I.~Gy{\"o}ngy, S.~Sabanis, and D.~{\v{S}}i{\v{s}}ka}, {\em Convergence of
  tamed {E}uler schemes for a class of stochastic evolution equations}, Stoch.
  Partial Differ. Equ. Anal. Comput., 4 (2016), pp.~225--245,
  \url{https://doi.org/10.1007/s40072-015-0057-7}.

%

\bibitem{2014arXiv1401.0295H}
{\sc M.~{Hutzenthaler} and A.~{Jentzen}}, {\em {On a perturbation theory and on
  strong convergence rates for stochastic ordinary and partial differential
  equations with non-globally monotone coefficients}}.
\newblock arXiv:1401.0295, 2014, \url{https://arxiv.org/abs/1401.0295}.

\bibitem{2016arXiv160402053H}
{\sc M.~{Hutzenthaler}, A.~{Jentzen}, and D.~{Salimova}}, {\em {Strong
  convergence of full-discrete nonlinearity-truncated accelerated exponential
  Euler-type approximations for stochastic Kuramoto-Sivashinsky equations}}.
\newblock arXiv:1604.02053, 2016, \url{https://arxiv.org/abs/1604.02053}.

\bibitem{JP}
{\sc A.~{Jentzen} and P.~{Pu{\v s}nik}}, {\em {Strong convergence rates for an
  explicit numerical approximation method for stochastic evolution equations
  with non-globally Lipschitz continuous nonlinearities}}.
\newblock arXiv:1504.03523, 2015, \url{https://arxiv.org/abs/1504.03523}.

\bibitem{MR3031667}
{\sc G.~T. Kossioris and G.~E. Zouraris}, {\em Finite element approximations
  for a linear fourth-order parabolic {SPDE} in two and three space dimensions
  with additive space-time white noise}, Appl. Numer. Math., 67 (2013),
  pp.~243--261, \url{https://doi.org/10.1016/j.apnum.2012.01.003}.

\bibitem{KLLAC}
{\sc M.~Kov{\'a}cs, S.~Larsson, and F.~Lindgren}, {\em On the backward {Euler}
  approximation of the stochastic {A}llen-{C}ahn equation.}, J. Appl. Probab.,
  52 (2015), pp.~323--338, \url{https://doi.org/10.1239/jap/1437658601}.

\bibitem{2015arXiv151003684K}
{\sc M.~{Kov{\'a}cs}, S.~{Larsson}, and F.~{Lindgren}}, {\em {On the
  discretisation in time of the stochastic Allen-Cahn equation}}.
\newblock arXiv:1510.03684, 2015, \url{https://arxiv.org/abs/1510.03684}. To appear in Matematische Nachrichten.

\bibitem{KLM}
{\sc M.~Kov{\'a}cs, S.~Larsson, and A.~Mesforush}, {\em Finite element
  approximation of the {C}ahn-{H}illiard-{C}ook equation}, SIAM J. Numer.
  Anal., 49 (2011), pp.~2407--2429, \url{https://doi.org/10.1137/110828150}.

\bibitem{MR3273327}
{\sc M.~Kov{\'a}cs, S.~Larsson, and A.~Mesforush}, {\em Erratum: {F}inite
  element approximation of the {C}ahn-{H}illiard-{C}ook equation}, SIAM J.
  Numer. Anal., 52 (2014), pp.~2594--2597,
  \url{https://doi.org/10.1137/140968161}.

%

\bibitem{Kur}
{\sc R.~Kurniawan}, {\em Numerical approximations of stochastic partial
  differential equations with non-globally {Lipschitz} continuous
  nonlinearities}.
\newblock Master thesis, ETH Z\"urich, 2014.

\bibitem{MR2846757}
{\sc S.~Larsson and A.~Mesforush}, {\em Finite-element approximation of the
  linearized {C}ahn-{H}illiard-{C}ook equation}, IMA J. Numer. Anal., 31
  (2011), pp.~1315--1333, \url{https://doi.org/10.1093/imanum/drq042}.

\bibitem{P2001}
{\sc J.~Printems}, {\em On the discretization in time of parabolic stochastic
  partial differential equations}, M2AN Math. Model. Numer. Anal., 35 (2001),
  pp.~1055--1078, \url{https://doi.org/10.1051/m2an:2001148}.

\bibitem{Thomeebook}
{\sc V.~Thom{\'e}e}, {\em {Galerkin Finite Element Methods for Parabolic
  Problems}}, Springer-Verlag, Berlin, 2006.

\end{thebibliography}

\def\cprime{$'$}

\end{document}